\documentclass[journal]{IEEEtran}  
\IEEEoverridecommandlockouts

\usepackage{xspace,amssymb,epsfig,syntonly}
\usepackage{epsfig,amsmath,color}
\usepackage{xspace,syntonly,empheq}
\usepackage{url}
\usepackage{bbm}
\usepackage{mathtools}
\usepackage{cite}
\usepackage{graphicx}
\usepackage{algorithm,algorithmic}
\usepackage{verbatim}
\usepackage{amssymb}
\usepackage{amsthm}
\usepackage{circuitikz}
\usepackage{enumerate}

\newcommand{\utwi}[1]{\mbox{\boldmath $#1$}}

\renewcommand{\hat}{\widehat}
\renewcommand{\tilde}{\widetilde}

\newcommand{\cD}{{\cal D}}

\newcommand{\cL}{{\cal{L}}}
\newcommand{\cN}{{\cal N}}

\newcommand{\cG}{{\cal G}}

\newcommand{\cE}{{\cal E}}

\newcommand{\cM}{{\cal M}}

\newcommand{\cX}{{\cal X}}
\newcommand{\cY}{{\cal X}}

\newcommand{\be}{{\bf e}}

\newcommand{\bg}{{\bf g}}

\newcommand{\bp}{{\bf p}}

\newcommand{\br}{{\bf r}}

\newcommand{\bx}{{\bf x}}

\newcommand{\bw}{{\bf w}}

\newcommand{\bz}{{\bf z}}
\newcommand{\by}{{\bf y}}

\newcommand{\bC}{{\bf C}}
\newcommand{\bD}{{\bf D}}

\newcommand{\bJ}{{\bf J}}

\newcommand{\bX}{{\bf X}}

\newcommand{\bbeta}{{\utwi{\beta}}}

\newcommand{\bgamma}{{\utwi{\gamma}}}

\newcommand{\blambda}{{\utwi{\lambda}}}

\newcommand{\bphi}{{\utwi{\phi}}}
\newcommand{\bPhi}{{\utwi{\Phi}}}

\newcommand{\reals}{\mathbb{R}}

\newcommand{\sfT}{\textsf{T}}

\newcommand{\proj}{\mathsf{Proj}}

\newcommand{\diam}{\mathsf{diam}}

\DeclarePairedDelimiterX{\norm}[1]{\lVert}{\rVert}{#1}

\usepackage{epstopdf}

\newtheorem{lemma}{Lemma}
\newtheorem{theorem}{Theorem}

\theoremstyle{definition}

\newtheorem{assumption}{Assumption}
\theoremstyle{definition}
\newtheorem{remark}{Remark}

\title{Online Primal-Dual Methods with Measurement Feedback for Time-Varying Convex Optimization} 

\author{Andrey Bernstein, Emiliano Dall'Anese, Andrea Simonetto$^\star$
\thanks{$^\star$Alphabetical order, authors contributed equally to the paper.}
\thanks{A. Bernstein is with the National Renewable Energy Laboratory (NREL), Golden, CO, USA. E. Dall'Anese is with the University of Colorado Boulder, Boulder, CO, USA. A. Simonetto is with the IBM Research Dublin, Dublin, Ireland. Emails: andrey.bernstein@nrel.gov, emiliano.dallanese@colorado.edu, andrea.simonetto@ibm.com.}
\thanks{This work was authored in part by the NREL, operated by Alliance for Sustainable Energy, LLC, for the U.S. Department of Energy (DOE) under Contract No. DE-AC36-08GO28308. This work was supported by the Laboratory Directed Research and Development (LDRD) Program at NREL. The views expressed in the article do not necessarily represent the views of the DOE or the U.S. Government. The U.S. Government retains and the publisher, by accepting the article for publication, acknowledges that the U.S. Government retains a nonexclusive, paid-up, irrevocable, worldwide license to publish or reproduce the published form of this work, or allow others to do so, for U.S. Government purposes.
}
}

\begin{document}
\maketitle

\begin{abstract}
This paper addresses the design and analysis of feedback-based online algorithms   to control systems or networked systems based on performance objectives and engineering constraints that may evolve over time. The emerging time-varying convex optimization formalism is leveraged to model optimal operational trajectories of the systems, as well as explicit local and network-level operational constraints. Departing from existing batch and feed-forward optimization approaches, the design of the algorithms capitalizes on an online implementation of primal-dual projected-gradient methods; the gradient steps are, however, suitably modified to accommodate  feedback from the system in the form of measurements -- hence, the term ``online optimization with feedback.'' By virtue of this approach, the resultant  algorithms can cope with model mismatches in the algebraic representation of the system states and outputs, they avoid pervasive measurements of exogenous inputs, and they naturally lend themselves to a distributed implementation. Under suitable  assumptions, analytical convergence claims are established in terms of dynamic regret. Furthermore, when the synthesis of the feedback-based online algorithms is based on a \emph{regularized} Lagrangian function, Q-linear convergence to solutions of the time-varying  optimization problem is shown. 
\end{abstract}

\section{Introduction}
\label{sec:introduction}

This paper focuses on time-varying optimization problems~\cite{Simonetto_Asil14} associated with systems or networked systems, for the purpose of modeling and controlling their operation  based on  performance objectives and engineering constraints that may evolve over time~\cite{opfPursuit,Rahili2015,Fazlyab2016a,Neely17}. The term ``networked systems'' here refers to a collection of systems coupled through intrinsic physical and behavioral interdependencies, and logically connected by an information infrastructure that supports given network-level control and optimization tasks. Examples include  communication systems, power grids, and robotic networks just to mention a few~\cite{Bullo18}. 

Suppose that  physical and/or behavioral interdependencies among systems in the network are modeled as 
\begin{align}
\label{eq:model_Phy}
\by(t) = \cM(\bx(t); t)
\end{align}
where $\bx(t) \in \mathbb{R}^n$ is a vector collecting given controllable inputs of the systems, $\by(t) \in \mathbb{R}^m$ represents observables or outputs of the network (quantities that pertain to both edges and nodes), and $\cM(\cdot \, ;  t): \mathbb{R}^n \rightarrow \mathbb{R}^m$ is a time-varying map defined over the domain of $\bx(t)$. For example, when a linear network model is utilized,~\eqref{eq:model_Phy} boils down to:
\begin{align}
\label{eq:model_PhyLin}
\by(t) = \bC \bx(t) + \bD \bw(t)
\end{align}
where $\bC \in \mathbb{R}^{m \times n}$ and $\bD \in \mathbb{R}^{m \times w}$ are given model parameters, and $\bw(t) \in \mathbb{R}^{w}$ is a vector of time-varying exogenous inputs (or, simply, uncontrollable quantities in the network). 

Consider associating with the networked systems a time-varying optimization of the form\footnote{\emph{Notation}: Upper-case (lower-case) boldface letters will be used for matrices (column vectors), and $(\cdot)^\sfT$ denotes transposition.  For a given $N \times 1$ vector $\bx \in \mathbb{R}^N$,  $\|\bx\|_2 := \sqrt{\bx^\sfT \bx}$.  Given a matrix $\bX \in \mathbb{R}^{N\times M}$, $[X]_{m,n}$  denotes its $(m,n)$-th entry and $\|\bX\|_2$ denotes the $\ell_2$-induced matrix norm. For a function $f: \mathbb{R}^N \rightarrow \mathbb{R}$, $\nabla_{\bx} f(\bx)$ returns the gradient vector of $f(\bx)$ with respect to $\bx \in \mathbb{R}^N$. $\mathrm{proj}_{\cX}\{\bx\}$ denotes a closest point to $\bx$ in $\cX$, namely $\mathrm{proj}_{\cX}\{\bx\} \in \arg \min_{\by \in \cX} \|\bx - \by\|_2 $.An operator $F: \cD \rightarrow \mathbb{R}^n$, $\cD \subseteq \mathbb{R}^n$ is strongly monotone with monotonicity constant $\eta$ if $(F(\bx) - F(\by))^\sfT(\bx - \by) \geq \eta \|\bx - \by\|^2$ for all $\bx, \by \in \cD$; the operator is monotone if $\eta = 0$. }
\begin{align} 
\label{eq:model_opt}
\min_{\bx \in \cX(t)} f(\bx, \by(\bx; t) ; t)
\end{align}
where $t \in \mathbb{R}_+$ is the temporal index; $\cX(t)$ is a convex set; $f: \mathbb{R}^{n} \times \mathbb{R}^{m} \times \mathbb{R}_+  \rightarrow \mathbb{R}$ is a convex function at each time $t$; and, the notation $\by(\bx; t)$ is utilized to stress that the observables $\by(t)$ depend on the vector variable $\bx$. The function $f$ is time-varying, in the sense that it can capture performance objectives that evolve over time. Accordingly, denoting as $\bx^{*}(t)$ an optimal solution of~\eqref{eq:model_opt} at time $t$, the optimization model~\eqref{eq:model_opt} leads to a continuous-time  optimal \emph{trajectory}. Given~\eqref{eq:model_Phy} and~\eqref{eq:model_opt}, the problem addressed in this paper pertains to the development and analysis of algorithms that enable \emph{tracking of an optimal trajectory} $\{\bx^{*}(t)\}_{t \in \mathbb{R}_+}$.  

For an isolated system or when the map~\eqref{eq:model_Phy} does not depend on time-varying exogenous inputs that are geographically and logically dispersed in the network, problem~\eqref{eq:model_opt} might be solved in a centralized setting based on a continuous time platform (see e.g.,~\cite{Arrowbook58,Cherukuri17, Rahili2015,Fazlyab2016a}); however, this paper focuses on the case where the measurements and communication of exogenous inputs introduce non-negligible delays, and the update of the input $\bx(t)$ leads to control actions that are implemented on digital control units. 

Let $s > 0$ denote a given sampling time and consider discretizing the solution trajectory of~\eqref{eq:model_opt} as $\{\bx^{*}(t_k)\}_{k \in \mathbb{N}}$, where $t_k := k s$. For perfect tracking,~\eqref{eq:model_opt} can be re-interpreted as a sequence of time-invariant problems that must be solved to convergence (i.e., batch solution) at each time $t_k$. However, a batch solution of~\eqref{eq:model_opt} might not be achievable within an interval that is consistent with the variability of $f(\cdot ; t)$ and the map $\cM(\cdot ; t)$ due to underlying communication and computational complexity requirements; for example, since iterative methods require multiple computation and communication rounds, the problem inputs $f(\cdot ; t)$ and $\cM(\cdot ; t)$ (and therefore the solution) might have already changed by the time the iterative method converged. Consider then the following \emph{online} first-order algorithm, tailored to the model~\eqref{eq:model_PhyLin} and to the case where the cost is $f(\by(\bx; t) ; t)$ for exposition simplicity:
\begin{align}
\bx(t_{k+1}) = & \proj_{\cX(t_k)}\Big\{ \bx(t_k) \nonumber \\
& \hspace{.5cm} - \alpha \bC^\sfT \nabla_\by f (\bC \bx(t_k) + \bD \bw(t_k); t_k)  \Big\} \label{eq:online}
\end{align}
where $\proj_{\cX}(\bz) := \arg \min_{\bx \in \cX} \|\bz - \bx\|_2$ denotes projection onto a convex set and $\alpha > 0$ is the step size. It is clear that $s$, in this case, represents the time required to perform one algorithmic iteration. 

Before elaborating on possible tracking properties of~\eqref{eq:online}, it is important to emphasize that the update~\eqref{eq:online} represents a \emph{feed-forward} (i.e., open loop) control method that presumes knowledge of the input-output map~\eqref{eq:model_PhyLin}. In fact, the function $f(\cdot ; t_k)$ in~\eqref{eq:online} is evaluated at the current output of the network, based on the postulated  model $\by(t_k) = \bC \bx(t_k) + \bD \bw(t_k)$. From a real-time optimization  perspective, this feature has fundamental drawbacks: 
\begin{itemize}
\item[(i)] The update~\eqref{eq:online} requires one to estimate the exogenous inputs $\bw(t_k)$ at each time $t_k$; this may be impractical in many realistic networked systems, especially when the number of exogenous inputs $w$ is much larger than $n$ and $m$ or when (part of) $\bw(t_k)$ might not be even observable. 
\item[(ii)] The feed-forward strategy~\eqref{eq:online} is sensitive to model mismatches; errors in the map~\eqref{eq:model_PhyLin} might drive the network operation to points that might not be implementable.   
\item[(iii)] The mathematical structure of the map $\cM(\bx(t); t)$ may prevent a distributed implementation of the update~\eqref{eq:online}.  
\item[(iv)] The update~\eqref{eq:online} does not acknowledge that the underlying systems may be governed by local controllers with given state dynamics; in fact,~\eqref{eq:online} presumes a time-scale separation where the local systems settle to a steady-state in response to a new command $\bx(t_k)$ within an interval $s$. 
\end{itemize}

To address challenges outlined above, the idea suggested in this paper is to  suitably modify the algorithmic updates of online optimization methods, such as~\eqref{eq:online}, to accommodate measurement feedback -- something that henceforth is referred to as \emph{online optimization with feedback}. In particular, letting $\hat{\bx}(t_k)$ and  $\hat{\by}(t_k)$ be \emph{measurements} of the input $\bx(t_k)$ and the output $\hat{\by}(t_k)$, respectively, we consider modifying~\eqref{eq:online} as 
\begin{align}
\bx(t_{k+1}) = & \proj_{\cX(t_k)}\Big\{ \hat{\bx}(t_k)  - \alpha \bC^\sfT \nabla_\by f (\hat{\by}(t_k); t_k)  \Big\} \label{eq:feedback}
\end{align}
where the measurement $\hat{\by}(t_k)$ replaces the network model $\by(t_k) = \bC \bx(t_k) + \bD \bw(t_k)$ and $\hat{\bx}(t_k)$ may replace the current iterate $\bx(t_k)$.  This simple conceptual modification leads to the following key advantages: 
\begin{itemize}
\item[(a.1)] Instead of measuring/estimating $w$ exogenous inputs $\bw(t_k)$,~\eqref{eq:feedback} relies on $m$ measurements of the outputs $\hat{\by}(t_k)$. This is of key importance when $m \ll w$.    
\item[(a.2)] The algorithm naturally accounts for the network physics via the measurements $\hat{\by}(t_k)$, and it does not rely on a synthetic network model.    
\item[(a.3)] The update~\eqref{eq:feedback} may naturally lend itself to a distributed implementation; see Remark \ref{re:distributed} in Section \ref{sec:problemformulation}. And,  
\item[(a.4)] The update~\eqref{eq:feedback} accounts for imperfect implementations/commands of the input $\bx(t_k)$ at the local systems. 
\end{itemize}

While the simplified setting~\eqref{eq:model_opt} and~\eqref{eq:feedback} was adopted to outline the main ideas, the following sections will present a much broader framework applicable to time-varying constrained convex problems. The design of the algorithms capitalizes on an online implementation of primal-dual projected-gradient methods; however, similar to~\eqref{eq:online}, the gradient steps are suitably modified to accommodate measurements. When the feedback-based primal-dual gradient method is applied to the time-varying Lagrangian, a dynamic regret analysis \cite{zinkevich} is provided. On the other hand, when considering a regularized Lagrangian function~\cite{Koshal11,opfPursuit,khuzani2016distributed}, performance of the proposed methods is assessed in terms of convergence of the iterates $\bx(t_{k})$ within a ball centered around the optimal trajectory $\{\bx^{*}(t_{k})\}_{k \in \mathbb{N}}$.  





This paper provides  the  following key contributions  relative to our domain-specific prior work~\cite{opfPursuit}: (i)~it considers generic time-varying  convex optimization problems with time-varying affine, linear, and nonlinear (convex) inequality constraints (on the other hand, \cite{opfPursuit} is limited to linear and affine inequality constraints); (ii) it provides a \emph{dynamic regret} analysis when a primal-dual gradient method is applied to the time-varying Lagrangian function; and (iii) it addresses the case where measurements of the network state are included in both primal and dual gradient steps, with due implications in the dynamic regret results as well as the \emph{Q-linear convergence} results obtained when considering a regularized Lagrangian function~\cite{Koshal11,khuzani2016distributed}. This paper also generalizes the domain-specific technical findings of~\cite{Tang17,Zhou17}, since~\cite{Zhou17} deals with linearly-constrained problems and~\cite{Tang17} leverages relaxations via approximate barrier functions. As a byproduct, the paper provides contributions over, e.g.,~\cite{Jokic09,Bolognani_feedback_15,Hirata,Hauswirth16}, where \emph{static} optimization problems were considered, and the earlier work~\cite{commelec1}  where no analytical convergence results were provided.  

In terms of  existing literature on regret analysis for online dual and primal-dual gradient methods \cite{onlineADMM,onlineSaddle,Lee17,TianyiChen2017,TianyiChen2017_Bandits}, the  contributions consist in: (i) proving dynamic (as opposed to  static) regret bounds; (ii) considering a general class of constrained optimization problems with feedback;  (iii) assuming time-varying feasible sets; and (iv) providing a bound on the average constraint violation. In particular, relative to~\cite{onlineSaddle}, the present paper considers primal-dual methods for generic time-varying constrained convex optimization problems (the analysis of~\cite{onlineSaddle} is limited to time-invariant consensus constraints), and projections in the algorithmic steps are performed on time-varying sets; further, the regret in~\cite{onlineSaddle} is computed with respect to a time-invariant optimizer, and no $Q$-linear convergence results are provided. With respect to the recent work \cite{TianyiChen2017}, the main contributions of this paper are: (i) the present paper addresses the design and analysis of algorithms for time-varying optimization problems where cost function, constraints, and implicit constraints $\bx(t) \in \cX(t)$ evolve over time (the implicit constraints must be satisfied at each iteration and are therefore not dualized in the construction of the Lagrangian); (ii) it provides linear convergence results when the algorithmic update is a strongly monotone operator; (iii) the results on the dynamic regret are derived under slightly weaker assumptions relative to\cite{TianyiChen2017}; and, (iv) the analysis of the primal-dual gradient method with errors due to measurements is a key novelty of the present paper.  

From an optimization standpoint, the paper extends the results of primal-dual-type methods of e.g.,~\cite{ OzdaglarSaddlePoint09,Koshal11,khuzani2016distributed,necoara2015linear} to the case of \emph{time-varying} problems and when feedback is utilized in the algorithmic steps [cf.~\eqref{eq:feedback}]. With respect to the time-varying problem formulations in~\cite{SimonettoGlobalsip2014,Ling14,Simonetto17}, the paper provides  results in the case of feedback-based methods. It is also worth pointing out that the proposed methodology can be cast within the domain of $\epsilon$-gradient methods~\cite{Bertsekas1999,Larsson03,necoara2014rate}; in this case, the paper extends the analysis of $\epsilon$-gradient methods to time-varying settings. 
Lastly, the paper provides an extensions of saddle-point flows~\cite{Arrowbook58,Elia-CDC11,Brunner-CDC12,Rahili2015,Cherukuri17} to the case of discrete-time steps, time-varying saddle functions, and feedback-based algorithmic steps.

The development of feedback-based online optimization methods has been, so far, driven by power systems application; see, for example, the works on frequency control~\cite{NaLi_ACC14,Papachristodoulou14} for transmission systems and for explicit power control in~\cite{opfPursuit,Tang17,Zhou17,commelec1,Hauswirth16}. However, the framework is generally applicable to a number of settings where the objective is to drive the  operation of physical and logical systems as well as networked systems to optimal operating points in real time. Application domains include, for example,  wireless communication systems~\cite{Low1999Flow,Chen12,Calvo18}, vehicle control~\cite{monteil2015},  water systems~\cite{Schutze03}, and robotic sensor networks~\cite{Bullo09}. 
It is also worth  pointing out that the general topic of online convex optimization and the associated (dynamic) regret analysis has been extensively studied in the theoretical machine learning literature; see, e.g., \cite{Besbes2015,Hall2015,Jadbabaie2015,Shahrampour2018} and references therein. This paper does not aim at providing a comprehensive overview of this topic; rather, the  main focus of this paper is to introduce the new concept of feedback-based online optimization, and to provide a regret analysis in the proposed setting as well as a convergence analysis in terms of the optimizer.

The remainder of the paper is organized as follows. Section~\ref{sec:problemformulation}  formulates the time-varying optimization problem and  outlines the proposed feedback-based online algorithm. Section~\ref{sec:regret}  provides a regret analysis for the algorithm when applied to the Lagrangian function, while Section~\ref{sec:convergenceSaddlePoint}  focuses on regularized Lagrangian functions. Section~\ref{sec:Examples} provides examples of applications, along with numerical results in Section~\ref{sec:results}.  Section~\ref{sec:Conclusions} concludes the paper.

\section{Feedback-Based Primal-Dual Method}
\label{sec:problemformulation}

Consider a network of $N$ systems, with the associated time-varying optimization problem:
\begin{subequations} 
\label{eqn:contProblem}
\begin{align} 
\textrm{(P0)}^{(t)} ~ &\min_{\substack{\bx \in \reals^n}}\hspace{.2cm} f_0(\by(\bx;t);t) + \sum_{i = 1}^N f_i(\bx_i;t) \\
&\mathrm{subject\,to:~} \bx_i \in \cX_i(t) , \, i = 1, \ldots, N \label{eqn:constr_X} \\
& \hspace{1.8cm} g_j(\by(\bx;t);t)  \leq 0 , \, j = 1, \ldots, M
 \end{align}
\end{subequations}
with $\cX_i(t) \subset \mathbb{R}^{n_i}$; $\sum_{i = 1}^N n_i = n$; and, where $\by(\bx; t) := \bC \bx + \bD \bw(t) \in \reals^m$ is  an algebraic representation of some observables in the systems as in~\eqref{eq:model_PhyLin}. Function $f_0(\by(\bx;t);t): \mathbb{R}^m \times \reals_+ \rightarrow \mathbb{R}$ is convex in $\bx$ at each time $t$, and it captures costs associated with the outputs $\by(\bx; t)$, while $f_i(\bx_i;t): \mathbb{R}^{n_i} \times \reals_+ \rightarrow \mathbb{R}$ is a convex function that models time-varying costs associated with the $i$-th sub-vector $\bx_i$. Finally, the convex functions $g_j(\by(\bx;t);t):  \mathbb{R}^{m} \times \reals_+ \rightarrow \mathbb{R}$ are utilized to impose time-varying constraints on $\by(\bx;t)$. We assume that $g_j(\by(\bx;t);t)$, for $j = 1, \ldots M_I$ is nonlinear and convex, whereas $g_j(\by(\bx;t);t)$, for $j = M_I + 1, \ldots M$, is linear or affine. 

As explained in the previous section, consider discretizing the temporal axis as $t_k = k s$, $k \in \mathbb{N}$, where $s > 0$ is a given sampling interval~\cite{Simonetto_Asil14,Simonetto17}\footnote{The choice of the sampling period is made depending on how fast one can run the computations (low $s$) and how much asymptotic error one can tolerate (high $s$).}. Accordingly, samples of the continuous-time problem~\eqref{eqn:contProblem} can be expressed as 
\begin{subequations} 
\label{eqn:sampledProblem}
\begin{align} 
\textrm{(P0)}^{(k)} ~ &\min_{\substack{\bx}}\hspace{.2cm} f_0^{(k)}(\by^{(k)}(\bx)) + \sum_{i = 1}^N f_i^{(k)}(\bx_i) \label{eq:obj_p0} \\
&\mathrm{subject\,to:~} \bx_i \in \cX_i^{(k)} , \, i = 1, \ldots, N \label{eqn:constr_Xsampl} \\
& \hspace{1.8cm} g_j^{(k)}(\by^{(k)}(\bx))  \leq 0 , \, j = 1, \ldots, M \label{eq:ineqconst}
 \end{align}
\end{subequations}
where $\cX_i^{(k)} := \cX_i(t_k)$, $f_i^{(k)}(\bx_i) := f_i(\bx_i;t_k)$, $\by^{(k)}(\bx) = \bC \bx + \bD \bw^{(k)}$, and similar notation is utilized for the remaining sampled quantities. 

For brevity, define $\bg^{(k)}(\by^{(k)}(\bx)) := [g_1^{(k)}(\by^{(k)}(\bx)), \ldots, g_M^{(k)}(\by^{(k)}(\bx))]^\sfT$, $f^{(k)}(\bx) := \sum_i f_i^{(k)}(\bx_i)$ and
\begin{align}
h^{(k)} (\bx) := f^{(k)}(\bx) + f_0^{(k)}(\by^{(k)} (\bx)). 
\end{align}
Further, 
let $\blambda \in \mathbb{R}_+^M$ denote the vector of dual variables associated with~\eqref{eq:ineqconst}. Then, the time-varying Lagrangian function is given by: 
\begin{align}
\cL^{(k)}(\bx, \blambda) & :=h^{(k)} (\bx) + \blambda^\sfT\bg^{(k)}\left(\by^{(k)}(\bx)\right) \, . \label{eqn:lagrangian}
\end{align}

Similar to, e.g.,~\cite{Koshal11}, consider the following regularized Lagrangian function
\begin{align} \label{eq:lagrangian_reg}
\cL^{(k)}_{p,d}(\bx, \blambda) := \cL^{(k)}(\bx, \blambda) + \frac{p}{2}\|\bx\|_2^2 - \frac{d}{2}\|\blambda\|_2^2 
\end{align}
where $p \geq 0$ and $d \geq 0$ are given regularization parameters, and consider  the following time-varying minimax problem: \begin{align} \label{eq:minmax}
\max_{\blambda \in \cD^{(k)}} \min_{\bx \in \cX^{(k)}} \cL^{(k)}_{p,d}(\bx, \blambda) \, \hspace{.5cm} k \in \mathbb{N}
\end{align}     
where $\cX^{(k)} := \cX_1^{(k)} \times \ldots \times \cX_N^{(k)}$ and $\cD^{(k)}$ is a convex and compact set constructed as explained shortly in Section~\ref{sec:regret} or as in~\cite{OzdaglarSaddlePoint09,Koshal11}. Hereafter,  $\bz^{(*, k)} := \{\bx^{(*,k)}, \blambda^{(*,k)}\}_{k\in \mathbb{N}}$ denote an \emph{optimal trajectory} of~\eqref{eq:minmax}. 

Based on the time-varying minimax problem~\eqref{eq:minmax}, the sequential execution of the following steps constitutes the proposed \emph{feedback-based online primal-dual gradient algorithm}: 
\begin{subequations} \label{eq:primalDualFeedback}
\begin{align}
\bx^{(k+1)} & = \proj_{\cX^{(k)}}\big\{(1 - \alpha p) \bx^{(k)} - \alpha \big(\nabla_\bx f^{(k)}(\bx^{(k)}) \nonumber \\  
& \hspace{-.5cm} + \bC^\sfT\nabla_\by f_0^{(k)}(\hat{\by}^{(k)} ) 
 + \sum_{j = 1}^M \lambda_j^{(k)}\bC^\sfT\nabla g_j^{(k)}(\hat{\by}^{(k)} )
\big)\big\}  \label{eq:primalstep} \\
\blambda^{(k+1)} & = \proj_{\cD^{(k)}}\left\{(1 - \alpha d) \blambda^{(k)} + \alpha \bg^{(k)}(\hat{\by}^{(k)} ) \right\} \label{eq:dualstep}
\end{align}
\end{subequations}
where $\alpha > 0$ is a constant step size, and $\hat{\by}^{(k)}$ is a measurement of $\by^{(k)}(\bx^{(k)})$ collected at time $t_k$. In the following, convergence results will be provided for the online algorithm~\eqref{eq:primalDualFeedback}, depending on the choice of the parameters $p$ and $d$. In particular, the following two cases are in order. 

\vspace{.1cm}

\noindent \textbf{Case~1}: $p = 0$, $d = 0$. Obviously, $\cL^{(k)}_{0,0}(\bx, \blambda) = \cL^{(k)}(\bx, \blambda)$, and $\{\bx^{(*,k)}\}_{k \in \mathbb{N}}$ is a (discretized) optimal solution trajectory of~\eqref{eqn:sampledProblem}. To capture the temporal variability of~\eqref{eq:minmax} (and, hence, of~\eqref{eqn:sampledProblem} as well as its continuous-time counterpart), define the following quantity:  
\begin{equation} \label{eqn:sigma}
\sigma^{(k)} := \|\bx^{(*, k+1)} - \bx^{(*, k)}\|_2 \, .
\end{equation}
Furthermore, let
\begin{equation} \label{eqn:regret_def}
R^{(k)} := \frac{1}{k} \sum_{\ell=1}^k \left(h^{(\ell)}(\bx^{(\ell)})  - h^{(\ell)}(\bx^{(*, \ell)}) \right)
\end{equation}
denote the \emph{average dynamic regret} at time step $t_k$. In this first case, to characterize the performance of the feedback-based online algorithm~\eqref{eq:primalDualFeedback}, asymptotic bounds on the dynamic regret $R^{(k)}$ will be established in Section~\ref{sec:regret}. Additionally, Section~\ref{sec:regret} will present asymptotic bounds on the \emph{average constraint violation.} The results for the dynamic regret and the average constraint violation are applicable also to other cases, where either $p$ or $d$ are positive.
	
\vspace{.1cm}

\noindent \textbf{Case~2}: $p > 0$, $d > 0$. In this case, the regularized Lagrangian $\cL^{(k)}_{p,d}(\bx, \blambda)$ is strongly convex in $\bx$ and strongly concave in the dual variables $\blambda$; hence, the  optimizer $\bz^{(*, k)}$ of~\eqref{eq:minmax}  is unique at each time $t_k$. The optimizer $\bz^{(*, k)}$, however, is not necessarily in the set of saddle points of the original Lagrangian $\cL^{(k)}(\bx, \blambda)$~\cite{Koshal11}. In fact, it is closely related to the so-called approximate Karush-Kuhn-Tucker (KKT) point~\cite{Andreani11} associated with the problem~\eqref{eqn:sampledProblem}; see, for example,~\cite{Koshal11} for a bound on the distance between $\bx^{(*,k)}$ and the solution of~\eqref{eqn:sampledProblem}.    In this case, asymptotic bounds will be derived for the Euclidean distance between $\bz^{(*, k)}$ and the output of the algorithm $\bz^{(k)} := \{ \bx^{(k)}, \blambda^{(k)}\}$; that is, the following quantity will be bounded: 
\begin{equation}
S^{(k)} := \|\bz^{(k)} - \bz^{(*, k)}\|_2 \, .
\end{equation}
Section~\ref{sec:convergenceSaddlePoint} will show that $S^{(k)}$ convergences $Q$-linearly within a ball centered about the optimal trajectory $\bz^{(*, k)}$. To derive bounds on $S^{(k)}$, the following quantity will be utilized to capture the temporal variability of the optimizer $\bz^{(*, k)}$ [cf.~\eqref{eqn:sigma}]:  
\begin{equation} \label{eqn:sigma2}
\bar{\sigma}^{(k)} := \|\bz^{(*, k+1)} - \bz^{(*, k)}\|_2 \, .
\end{equation}

\vspace{.1cm}

It is worth mentioning that the dynamic regret analysis could be applicable also to \textbf{Case 2}; however, the resultant regret-type results would be with respect to a perturbed solution that one has by utilizing the regularized Lagrangian function. The objective of \textbf{Case 2} is to show that, by utilizing a regularized Lagrangian function, one can  establish $Q$-linear convergence associated with $S^{(k)}$. These two cases highlight the different convergence results that become available based on the choice of the parameters $p$ and $d$.  

For exposition simplicity, the paper focuses on the case where only measurements of $\by^{(k)}(\bx^{k})$ are utilized in the steps~\eqref{eq:primalDualFeedback}; however, the results can be naturally extended to the case where measurements of $\bx^{(k)}$ are utilized too.

Pertinent assumptions that are utilized to derive the results explained above are stated next.

\begin{assumption} 
\label{ass:constqualification}
Slater's constraint qualification holds at each time instant $k$. 
\end{assumption}
\begin{assumption} \label{ass:sets}
The set $\cX^{(k)}$ is convex and compact for all $k$. Moreover, the sequence $\{\cX^{(k)}\}$ is \emph{uniformly} bounded. 
That is, $B := \sup_{k \geq 1} \sup_{\bx \in \cX^{(k)}} \|\bx\|_2 < \infty$. Also, let $D < \infty$ denote the upper bound on the diameters of $\{\cY^{(k)}\}$, so that $\diam (\cY^{(k)}) \leq D$ for all $k$.
\end{assumption}
\begin{assumption} 
\label{ass:cost}
The functions $f_0^{(k)}(\by)$ and $f^{(k)}_i(\bx_i)$ are convex and continuously differentiable for all $k$. The gradient map $\nabla_\bx f^{(k)}(\bx)$ is Lipschitz continuous with constant $L \geq 0$ over $\reals^n$ for all $k$.   Furthermore, $\nabla_\by f_0^{(k)}(\by)$  is Lipschitz continuous with constant $L_0 \geq 0$ over $\reals^m$ for all $k$. 
\end{assumption}
\begin{assumption} 
\label{ass:nonlinearconstr}
For each $j = 1, \ldots, M_I$ and all $k$, the function $g_j^{(k)}(\by)$ is convex and continuously differentiable.  Moreover, it has a  Lipschitz continuous gradient with constant $L_{g_m} > 0$. Let $\bJ^{(k)} (\by)$ denote the Jacobian (matrix-valued) map of $\bg^{(k)}(\by)$ with entries
\begin{equation}
\left(\bJ^{(k)} (\by) \right)_{i \ell} := \frac{\partial (\bg^{(k)}(\by))_i }{ \partial (\by)_{\ell}},
\end{equation}
and let $L_{G} \geq 0$ denote the Lipschitz constant of $\bJ^{(k)} (\by)$.
\end{assumption}
\begin{assumption} \label{ass:error}
There exists a scalar $e_y < + \infty$ such that the measurement error can be bounded as  
\begin{align}\sup_{k \geq 1} \|\hat{\by}^{(k)} - \by^{(k)}(\bx^{(k)})\|_2 \leq e_y.
\end{align}
\end{assumption}

From Assumption~\ref{ass:constqualification}, it follows that strong duality holds uniformly in time for the convex problems~\eqref{eqn:sampledProblem}. It is worth noticing that,  from the continuity of the Jacobian and the compactness of $\cX^{(k)}$,  there exists a scalar $M_g < + \infty$ such that $\|\bJ^{(k)} (\by)\|_2  \leq M_g$ for all $k$. In fact, one can set:
\begin{align}
M_g = \sup_{k} \max_{\bx \in \cX^{(k)}} \|\bJ^{(k)} (\by)\|_2. 
\end{align}
Then, using the Mean Value Theorem, one can show that 
\begin{align} \label{eq:lip_g}
\|\bg^{(k)}(\by_1) - \bg^{(k)}(\by_2)\|_2 \leq M_g \|\by_1 - \by_2\|_2
\end{align}
for all $k \in \mathbb{N}$. The parameter $M_g$ will be utilized in the subsequent sections to establish various convergence results.  Since the online algorithm~\eqref{eq:primalDualFeedback} leverages measurements of $\by^{(k)}(\bx^{(k)})$ at each time $k \in \mathbb{N}$, the bound in Assumption~\ref{ass:error} models measurements errors, quantization errors, model mismatches between the network physics and  the algebraic representation~\eqref{eq:model_PhyLin}, and imperfect implementation of the input $\bx(t_k)$ at the local systems/nodes (that translates into an imperfect $\by^{(k)}(\bx^{(k)})$).

With these assumptions in place, a dynamic regret analysis will be presented in the ensuing section. Per-iteration and asymptotic bounds on $S^{(k)}$ will then be presented in Section~\ref{sec:convergenceSaddlePoint}. But first, a remark on the distributed implementation is in order. 

\begin{remark}[Distributed implementation]
\label{re:distributed}
Similarly to the illustrative example~\eqref{eq:online}, the model-based counterpart of~\eqref{eq:primalDualFeedback} requires a centralized implementation of the primal and dual projected gradient steps. In fact, the iterates $\{\bx_i^{(t)}\}_{i = 1}^N$ pertaining to the $N$ systems (or a subset of them, depending on the zero entries  in the matrix $\bC$) must be collected at a fusion center or network-level controller in order to evaluate the gradient $\nabla f_0^{(t)}(\bC \bx^{(k)} + \bD \bw^{(k)})$ and the gradient $\nabla g_i^{(t)}(\bC \bx^{(k)} + \bD \bw^{(k)})$ in the primal update, as well as the function $\bg^{(t)}(\bC \bx^{(k)} + \bD \bw^{(k)})$ in the dual update. On the other hand, the measurement-based steps~\eqref{eq:primalDualFeedback} naturally decouple into $N$ updates, where each system $i$ updates $\bx_i^{(t)}$ locally; the dual step can be performed locally by  sensors or by network agents, which subsequently  broadcast $\nabla f_0^{(t)}(\by^{(t)})$ and $\lambda_j^{(k)}\bC^\sfT\nabla g_j^{(k)}(\hat{\by}^{(k)} )$ to the $N$ systems.  
\end{remark}


\section{Regret Analysis}
\label{sec:regret}

Recall that in \textbf{Case 1} the regularization parameters are $p = d = 0$. Then, since $\cL^{(k)}_{0,0}(\bx, \blambda) = \cL^{(k)}(\bx, \blambda)$, the primal update  \eqref{eq:primalstep} can be compactly re-written as
\begin{align}
\bx^{(k+1)} & = \proj_{\cX^{(k)}}\big\{\bx^{(k)} - \alpha \hat{ \nabla}_x \cL^{(k)} (\bx^{(k)}, \blambda^{(k)}) \big \}, \label{eq:primalstep_regret}
\end{align}
where
\begin{align}
 \hat{ \nabla}_x \cL^{(k)} (\bx, \blambda) &:= \nabla f^{(k)} (\bx) + \bC^\sfT \nabla f_0^{(k)} (\hat{\by}^{(k)}) \nonumber \\
 & \qquad + \left( \bJ^{(k)} (\hat{\by}^{(k)}) \bC \right)^\sfT \blambda \label{eq:gradXApprox}
 \end{align}
On the other hand, the sets $\cD^{(k)}$ in \eqref{eq:dualstep} are  chosen as follows: 
\begin{equation} \label{eq:Lambda}
\cD^{(k)} \equiv \Lambda_{\alpha, \kappa} := \left\{\blambda \in \reals_{+}^M: \, \|\blambda\|_2 \leq \frac{1}{\alpha^\kappa}\right\}
\end{equation}
for some $\kappa >0$.

A similar choice of $\cD^{(k)}$ can be utilized in Section~\ref{sec:convergenceSaddlePoint}; however, the  choice of \eqref{eq:Lambda} is particularly essential for the regret bounds below.

The dynamic regret of the algorithm is analyzed next. To this end, introduce the following  notation for brevity: 
\begin{align}
G & := \|\bC\|_2 M_g, \\ 
F &:= \sup_{k \geq 1} \sup_{\bx \in \cX^{(k)}} \|\nabla h^{(k)} (\bx) \|_2,\\
g &:= \sup_{k \geq 1} \sup_{\bx \in \cX^{(k)}} \| \bg^{(k)} (\by^{(k)} (\bx))\|_2, \\
L_x &:= \|\bC\|_2 \left(L_0 + \frac{L_G}{\alpha^\kappa}  \right), \\
F_x &:= F + \frac{G}{\alpha^\kappa}.
\end{align} 
With this notation in place, the following results for the dynamic regret and constraint violation are presented; the proofs are provided in Appendix \ref{sec:proof_regret}. 

\begin{theorem} \label{thm:regret}
Under Assumptions \ref{ass:constqualification}, \ref{ass:sets}, \ref{ass:cost}, \ref{ass:nonlinearconstr}, and \ref{ass:error}, for any $\alpha > 0$, $\kappa > 0$, and $k \in \mathbb{N}$, we have that
\begin{align}
&R^{(k)} \leq B^{(k)}(\alpha, \kappa) := \frac{B}{\alpha k} +  K_1 \alpha + K_2 \alpha^{1 - \kappa} + K_3 \alpha^{1 - 2 \kappa} \nonumber\\
&\quad  + K_4(\alpha) e_y +  K_5(\alpha) e_y^2 + K_6(\alpha) \frac{1}{k}\sum_{\ell=1}^k \sigma^{(\ell)} \label{eqn:regret_anytime}  
\end{align}
and, therefore,
\begin{align}
&\limsup_{k \rightarrow \infty} R^{(k)} \leq B^{(\infty)}(\alpha, \kappa) :=  K_1 \alpha + K_2 \alpha^{1 - \kappa} + K_3 \alpha^{1 - 2 \kappa} \nonumber\\
&\quad  + K_4(\alpha) e_y +  K_5(\alpha) e_y^2 + K_6(\alpha) \limsup_{k \rightarrow \infty}  \frac{1}{k}\sum_{\ell=1}^k \sigma^{(\ell)} \label{eqn:regret}
\end{align}
where 
\begin{align*}
&K_1 := \frac{1}{2}(F^2 + g^2), \, K_2 := FG, \, K_3 := \frac{1}{2} G^2 \\
&K_4(\alpha) := (2B + \alpha F + \alpha^{1 - \kappa} G) L_x + \left(\frac{2}{\alpha^\kappa} + \alpha g \right ) M_g, \\
&K_5(\alpha) := \frac{\alpha}{2} (L_x^2 + M_g^2), \\
&K_6(\alpha) := \frac{1}{\alpha}(D + B).
\end{align*}
\end{theorem}

\begin{theorem} \label{thm:violation}
Under Assumptions \ref{ass:constqualification}, \ref{ass:sets}, \ref{ass:cost}, \ref{ass:nonlinearconstr}, and \ref{ass:error}, for any $\alpha > 0$, $\kappa > 0$, and $k \in \mathbb{N}$, the constraint violation induced by the algorithm~\eqref{eq:primalDualFeedback} can be bounded as:
\begin{subequations} \label{eqn:constraints_anytime}
\begin{align} 
& \frac{1}{k} \sum_{\ell=1}^k \bg^{(\ell)}(\by^{(\ell)}(\bx^{(\ell)})) \nonumber \\
& \hspace{2.0cm} \leq \alpha^\kappa \left(B^{(k)}(\alpha, \kappa) - R^{(k)}\right)\\
& \hspace{2.0cm} \leq \alpha^\kappa (B^{(k)}(\alpha, \kappa) + 2 F B), \label{eqn:constraints_anytime2}
\end{align}
\end{subequations}
and, therefore, 
\begin{subequations} \label{eqn:constraints}
\begin{align} 
& \limsup_{k \rightarrow \infty} \frac{1}{k} \sum_{\ell=1}^k \bg^{(\ell)}(\by^{(\ell)}(\bx^{(\ell)})) \nonumber \\
& \hspace{2.0cm} \leq \alpha^\kappa \left(B^{(\infty)}(\alpha, \kappa) - \liminf_{k \rightarrow \infty} R^{(k)}\right)\\
& \hspace{2.0cm} \leq \alpha^\kappa (B^{(\infty)}(\alpha, \kappa) + 2 F B), \label{eqn:constraints2}
\end{align}
\end{subequations}
where the $\limsup$ and  inequality are component-wise; and $B^{(k)}(\alpha, \kappa)$ and $B^{(\infty)}(\alpha, \kappa)$ are given in \eqref{eqn:regret_anytime} and \eqref{eqn:regret}, respectively.
\end{theorem}

The following remarks are in order. 

\begin{remark} \label{rem:Dk}
It is a standard procedure in the analysis of primal-dual methods to establish boundedness of the dual iterates and to  project the dual variables onto a compact set containing the optimal ones; see, for example,~\cite{Koshal11,OzdaglarSaddlePoint09}.  A  different approach was taken in~\cite{TianyiChen2017}, where additional assumptions were imposed to guarantee the boundedness of the dual iterates to prove regret bounds. We defined the set $\cD^{(k)}$ as in~\eqref{eq:Lambda} for mathematical tractability, and to obtain sharper dynamic regret and constraint violation bounds. We note that for the $Q$-linear convergence analysis of Section \ref{sec:convergenceSaddlePoint}, it is necessary for $\cD^{(k)}$ to include the optimal dual variables; on the other hand, the regret analysis of this section
does not require $\cD^{(k)}$ to include the optimal dual variables. Relative to~\cite{TianyiChen2017},~\eqref{eq:Lambda} has merits because it allows to avoid additional assumptions on the constraint functions.     
\end{remark}

\begin{remark}
Since the right-hand-side of \eqref{eqn:constraints_anytime2} and \eqref{eqn:constraints2} is positive, the bounds of Theorem \ref{thm:violation} also apply to the actual average constraint violation, namely to
\[
\max  \left \{0, \frac{1}{k} \sum_{\ell=1}^k \bg^{(\ell)}(\by^{(\ell)}(\bx^{(\ell)}))  \right\};
\]
see \cite{TianyiChen2017} for a similar definition.
\end{remark}

\begin{remark} \label{rem:kappa}
The optimal choice of the parameter $\kappa$ is in general hard to obtain due to the complicated dependency of the terms $K_4(\alpha)$ and $K_5(\alpha)$ on $\kappa$. Ignoring the terms corresponding to $e_y$, the optimal choice is $\kappa = \frac{1}{3}$. Indeed, the dominating term in \eqref{eqn:regret} is $K_3 \alpha^{1 - 2 \kappa}$, and the dominating term in \eqref{eqn:constraints} is $2 F B \alpha^\kappa$. Therefore, asymptotically, for $\alpha \rightarrow 0$, the optimum is obtained when $1 - 2 \kappa = \kappa$. 
\end{remark}

\begin{remark} \label{rem:comp}
The definition of dynamic regret utilized in \eqref{eqn:regret_def} is with respect to the optimal sequence $\{\bx^{(*, k)} \}$. However, the results of Theorems \ref{thm:regret} and \ref{thm:violation} hold also for \emph{any} comparator (or reference) sequence $\{\tilde{\bx}^{(k)} \}$, as is for example in \cite{zinkevich}. In that case, $\sigma^{(k)} := \|\tilde{\bx}^{(k+1)} - \tilde{\bx}^{(k)} \|_2 $ captures the temporal variability of the comparator sequence.
\end{remark}

\begin{remark} \label{rem:bounds}
Note that in the error-free case ($e_y = 0$) and when the variability of the comparator sequence [cf.~Remark \ref{rem:comp}] is bounded, namely
\[
\sum_{\ell=1}^k \sigma^{(\ell)} \leq B_\sigma, \, \forall k
\]
for some $B_\sigma < \infty$, the obtained results are similar in spirit to the classical dynamic regret bounds (e.g., in \cite{zinkevich}). In particular, taking $\kappa = \frac{1}{3}$ as in Remark \ref{rem:kappa}, it follows from \eqref{eqn:regret_anytime} and \eqref{eqn:constraints_anytime} that
\begin{align*}
&R^{(k)} \leq \frac{1}{\alpha k} (B + B_\sigma (D + B)) + K_1 \alpha + K_2 \alpha^{\frac{2}{3}} + K_3 \alpha^{\frac{1}{3}}
\end{align*}
\begin{align*}
&\frac{1}{k} \sum_{\ell=1}^k \bg^{(\ell)}(\by^{(\ell)}(\bx^{(\ell)})) \leq \frac{1}{\alpha^{\frac{2}{3}} k} (B + B_\sigma (D + B)) \\
&\qquad \qquad \qquad + K_1 \alpha^{\frac{4}{3}} + K_2 \alpha + K_3 \alpha^{\frac{2}{3}} + 2 F B\alpha^{\frac{1}{3}}.
\end{align*}
Therefore, using a standard choice of 
$\alpha := \frac{1}{k^\beta}$ for some $0 < \beta < 1$, one would obtain
\begin{align*}
&R^{(k)} \leq \frac{1}{k^{1-\beta}} (B + B_\sigma (D + B)) + \frac{K_1}{k^\beta} + \frac{K_2} {k^{\frac{2}{3} \beta}} + \frac{K_3}{k^{\frac{1}{3} \beta}}, \\
&\frac{1}{k} \sum_{\ell=1}^k \bg^{(\ell)}(\by^{(\ell)}(\bx^{(\ell)})) \leq \frac{1}{k^{1 - \frac{2}{3} \beta}} (B + B_\sigma (D + B)) \\
&\qquad \qquad \qquad + \frac{K_1}{k^{\frac{4}{3} \beta}} + \frac{K_2}{k^\beta} +\frac{K_3}{k^{\frac{2}{3} \beta}} + \frac{2 F B}{k^{\frac{1}{3} \beta}}.
\end{align*}
Since the dominating terms are $\frac{1}{k^{1 - \beta}}$ and $\frac{1}{k^{\frac{1}{3} \beta}}$, if $\beta$ satisfies $1 - \beta = \frac{1}{3} \beta$, (namely $\beta = \frac{3}{4}$), then one achieves the best convergence rate \emph{simultaneously} for dynamic regret and constraint violation of $O\big(1 / k^{\frac{1}{4}}\big)$. Note that this convergence rate is inferior to the optimal regret bound of $O\big(1 / \sqrt{k} \big)$ known in the literature for the standard online convex optimization algorithms.
There is an evidence that by modifying the primal-dual algorithm and imposing slightly stronger assumptions, the optimal regret bound can be obtained for the error-free case.  For example, it was recently shown in \cite{TianyiChen2017}, that a modified primal-dual algorithm (with modified primal step) leads to optimal regret bounds under some additional assumptions on the constraint function (see Theorem 1 and the requirement on the denominator in equation (11) in \cite{TianyiChen2017}). We would like to stress that our main goal here is to introduce the concept of ``closed-loop'' optimization of systems (via measurement feedback), and analyze the most natural algorithm \emph{both in terms of regret and Q-linear convergence of the optimizer.} Modifying the algorithm and assumptions to optimize the regret bound might not necessarily lead to an algorithm with better asymptotic error for the time-varying scenario. In any case, these questions remain a subject of future research.

\end{remark}

The ensuing section will consider the case of regularized Larangian functions.

\section{Tracking of Time-varying Saddle Points}
\label{sec:convergenceSaddlePoint}

Let $p > 0$ and $d > 0$, and consider re-writing the algorithmic steps~\eqref{eq:primalDualFeedback} in the following compact form:
\begin{align}
\bz^{(k+1)} & = \proj_{\cX^{(k)} \times \cD^{(k)}}\big\{\bz^{(k)} - \alpha \hat{\bphi} (\bz^{(k)}) \big \}, \label{eq:step_tracking}
\end{align}
with the time-varying map $\hat{\bphi}: \cX^{(k)} \times  \cD^{(k)}  \rightarrow \mathbb{R}^n \times \mathbb{R}^M$ is defined as 
\begin{equation}
\label{phi_mapping}
 \hat{\bphi}^{(k)}: \bz \mapsto 
 \left[\begin{array}{c}
 \hat{\nabla}_{\bx}  \cL_{p,d}^{(k)}(\bz) \\
 - \bg^{(k)}(\hat{\by}^{(k)}) + d \blambda 
\end{array}
\right],
\end{equation}
and where, similarly to~\eqref{eq:gradXApprox}, $\hat{\nabla}_{\bx}  \cL_{p,d}^{(k)}(\bz)$ is the approximate gradient of the regularized Lagrangian function calculated as:
\begin{align}
 \hat{ \nabla}_\bx \cL_{p,d}^{(k)} (\bz) & = \hat{ \nabla}_\bx \cL^{(k)} (\bz) - p \bx  \label{eq:gradLambdaApprox2}
 \end{align}

Similar to the previous section, let $\nabla_\bx \cL_{p,d}^{(k)} (\bz) = \nabla_\bx \cL^{(k)} (\bz) - p \bx$ be the gradient of the regularized Lagrangian evaluated at $\bz := (\bx, \blambda)$ and at the the synthetic output $\by^{(k)}(\bx)$. Using $\nabla_\bx \cL_{p,d}^{(k)} (\bz)$, let $\bphi^{(k)}: \cX^{(k)} \times  \cD^{(k)}  \rightarrow \mathbb{R}^n \times \mathbb{R}^M$ be the counterpart of $\hat{\bphi}^{(k)}$ when the model $\by^{(k)}(\bx)$ is utilized; that is, 
\begin{equation}
\label{phi_mapping_true}
\bphi^{(k)}: \bz \mapsto 
 \left[\begin{array}{c}
 \nabla_{\bx}  \cL_{p,d}^{(k)}(\bz) \\
 - \bg^{(k)}(\by^{(k)}(\bx)) + d \blambda 
\end{array}
\right],
\end{equation}
Replacing $\hat{\bphi}^{(k)}$ with $\bphi^{(k)}$ in~\eqref{eq:step_tracking} yields a \emph{feed-forward} online algorithm, as discussed in Section~\ref{sec:introduction}. 

Recall that $\bar{\sigma}^{(k)} := \|\bz^{(*, k+1)} - \bz^{(*, k)}\|_2$. The main results are stated next.

\begin{theorem}
\label{theorem.inexact}
Consider the sequence $\{\bz^{(k)}\}$ generated by the algorithm~\eqref{eq:primalstep}--\eqref{eq:dualstep}. Under Assumptions \ref{ass:constqualification}, \ref{ass:sets}, \ref{ass:cost}, \ref{ass:nonlinearconstr}, and \ref{ass:error},  the distance between $\bz^{(k)}$ and the optimizer $\bz^{(*,k)}$ of~\eqref{eq:minmax} at time $k$ can be bounded as:
\begin{align} 
\|\bz^{(k)} - \bz^{(*,k)}\|_2 \leq & c(\alpha)^k \|\bz^{(0)} - \bz^{(*,0)}\|_2 + \sum_{\ell = 0}^{k-1} c(\alpha)^\ell \bar{\sigma}^{(k-\ell-1)}  \nonumber \\ 
& \hspace{-2.5cm} + \sum_{\ell = 0}^{k-1} c(\alpha)^\ell \left(\alpha \|\bphi(\bz^{(k-\ell-1)}) - \hat{\bphi} (\bz^{(k-\ell-1)}) \|_2 \right)
\label{eqn:inst_bound}
\end{align}
with $c(\alpha)$ given by 
\begin{align}
c(\alpha) := & [1 - 2 \alpha \eta_\phi + \alpha^2 L_\phi^2]^\frac{1}{2}  \label{eq:kappa}
\end{align} 
where 
\begin{align}
\eta_{\phi} & := \min\{p, d\} \\
L_\phi & := \sqrt{(L + p + M_g + \xi_\lambda L_g)^2 + (M_g + d)^2} 
\end{align}
and $L_g := \sqrt{\sum_{m = 1}^{M_I} L_{g_m}^2}$, $\xi_\lambda := \sup_k \max_{\blambda \in \cD^{(k)}} \|\blambda\|_2$.
\end{theorem}
\emph{Proof.} See Appendix~\ref{sec:proof_linear}.

\vspace{.2cm}

Theorem~\ref{theorem.inexact} provides a bound on $\|\bz^{(k)} - \bz^{(*,k)}\|_2$ per each time instant $k \in \mathbb{N}$. Asymptotic bounds are established next.  

But first, notice that the term $\|\bphi(\bz^{(k-\ell-1)}) - \hat{\bphi} (\bz^{(k-\ell-1)}) \|_2$ is due to the errors in the computation of gradients~\cite{Bertsekas1999,Larsson03} that one commits by ``closing the loop''; i.e., by replacing the model $\by^{(k)}(\bx^{(k)})$ with the measurements $\hat{\by}^{(k)}$. The term $\|\bphi(\bz^{(k-\ell-1)}) - \hat{\bphi} (\bz^{(k-\ell-1)}) \|_2$ is shown to be bounded next.

\begin{theorem} \label{cor:asym}
Suppose that there exists a scalar $\bar{\sigma} < + \infty$ such that $\sup_{k \geq 1} \bar{\sigma}^{(k)} \leq \bar{\sigma}$, and $c(\alpha) < 1$. Then, the sequence  $\{\bz^{(k)}\}$ converges Q-linearly to $\{\bz^{(*,k)}\}$ up to an asymptotic error bound given by:
\begin{align} 
\limsup_{k\to\infty} \|\bz^{(k)} - \bz^{(*, k)}\|_2 &\leq \frac{\alpha \sqrt{e_p^2 + e_d^2} + \bar{\sigma}}{1 - c(\alpha)}  \label{eqn:asym_bound_apr}
\end{align}
where 
\begin{align}
e_p & \leq \left(L_0 + M_\lambda M_I \max_{j = 1, \ldots, M_I}\{L_{g_j}\}\right) \|\bC\|_2 e_y \\
e_d & \leq M_g e_y \, ,
\end{align}
with $M_{\lambda} := \sup_{k \geq 1} \max_{\blambda \in \cD^{(k)}} \|\blambda\|_1$\, . 
\end{theorem}
\emph{Proof.} See Appendix~\ref{sec:proof_linear}.

\vspace{.2cm}

The coefficient $c(\alpha)$ is less then one when $\alpha < 2 \eta_\phi / L_\phi^{2}$. When no measurement errors are present, $e_p = e_d = 0$ and~\eqref{eqn:asym_bound_apr} provides a result for feed-forward online algorithms (similar to e.g.,~\cite{SimonettoGlobalsip2014,Simonetto17}). When $\bar{\sigma} = 0$, then the underlying optimization problem is static  and the algorithm converges to the solution of the static optimization problem~\eqref{eq:minmax}. Finally, notice that the result~\eqref{eqn:asym_bound_apr}  can also be interpreted as input-to-state stability result, where the optimal trajectory $\{\bz^{(*,k)}\}$ of the time-varying minimax problem~\eqref{eq:minmax} is taken as a reference.

The results rely on the fact that the map $\bphi^{(k)}(\bz)$ is strongly monotone over $\cX^{(k)} \times \cD^{(k)}$ with constant $\eta_{\phi}$ and Lipschitz over $\cX^{(k)} \times \cD^{(k)}$ with coefficient $L_\phi$. A discussion on the cases where $\bphi^{(k)}$ is strongly monotone and Lipschitz follows:
\begin{enumerate}[(i)]
\item Suppose that the function $h^{(k)}$ is convex but  not strongly convex; suppose further that  $p > 0$ and $d > 0$. Then, $\bphi^{(k)}$ is strongly monotone and Lipschitz.
\item When $d > 0$, $p = 0$, and the function $h^{(k)}$ is strongly convex, it is easy to show that $\bphi^{(k)}$ is strongly monotone and Lipschitz. 
\item However, the map  $\bphi^{(k)}(\bz)$  is \emph{not} strongly monotone  in the following cases: iii.a) $d = 0$, irrespective of $p$ and $h^{(k)}$; and, when iii.b) $d > 0$, $p = 0$ and $h^{(k)}$ is not strongly convex. Therefore, strong monotonicity is not  present in \textbf{Case 1}. Strong monotonicity is a key property of the maps $\hat{\bphi}^{(k)}(\bz)$ and $\bphi^{(k)}(\bz)$ that is utilized in the Q-linear convergence analysis; this explains why this paper provided only a regret analysis for \textbf{Case 1}.  
\end{enumerate}

There always  exists a scalar $\bar{\sigma} < + \infty$, since it is assumed that the sets $\cX^{(k)}$ and $\cD^{(k)}$ are compact uniformly in time (and, therefore, optimal solutions are never unbounded) and the solution set is not empty. For $s \rightarrow 0$, $\bar{\sigma}$ is in fact an upper bound on the norm of the gradient of the optimal trajectory $\{\bz^{(*,k)}\}_{k \in \mathbb{N}}$. 

Lastly, regarding the  values of $p$ and $d$, they should  be selected numerically (or analytically, whenever possible) based on specific implementation goals. For example, larger values of $p$ and $d$ lead to larger perturbations of the solution trajectories of the original time-varying optimization problem; therefore, in order not to sacrifice optimality, $p$ and $d$ should be selected as small as possible, while ensuring that $c(\alpha) < 1$. Small values of $p$ and $d$, however, make $c(\alpha)$  close to $1$ [cf.~\eqref{eq:kappa}], thus involving a larger asymptotic bound~\eqref{eqn:asym_bound_apr}. On the other hand, minimizing the asymptotic bound~\eqref{eqn:asym_bound_apr} requires larger values of $p$ and $d$, thus  sacrificing optimality and constraint satisfaction.

\section{Examples of Applications}
\label{sec:Examples}

In this section, two examples of applications of the proposed framework will be outlined.

\subsection{Example in Communication Systems}
\label{sec:comm}

An illustrative example is provided for the flow control problem in communications systems~\cite{Low1999Flow}; the proposed approach can  also be applied to stochastic routing problems and energy-harvesting communication networks~\cite{Calvo18}. Consider a communication network modeled as a directed graph $\cG = (\cN, \cE)$, with $\cN$ the set of nodes and $\cE$ the set of directed edges, which are dictated by (possibly time varying) routing matrix. For a given node $i$, the set of nodes that forward traffic to the $i$-th one is denoted as $\cN_i^\textrm{in} := \{j: (j,i) \in \cE\}$; similarly, define the set $\cN_i^\textrm{out} := \{n: (i,n) \in \cE\}$.  Consider $S \in \mathbb{N}$ traffic flows and let $r_{ij,s} \geq 0$ be the flow on link $(i,j) \in \cE$ for the $s$-th flow and $x_{i,s}$ the traffic generated by a source node ($x_{i,s} \geq 0$) or delivered at a destination node $i$ ($x_{i,s} \leq 0$) for flow $s$.  Let $c^{(k)}_{ij}(p_{ij})$ be the time-varying capacity of link $(i,j) \in \cE$ as a function of the transmission power $p_{ij}$, and $w_{ij}^{(k)}$ the exogenous (i.e., uncontrollable) traffic on the same communication link. 

Consider then the following time-varying problem to maximize the traffic generated (and delivered) and to compute the communication flows:      
\begin{subequations} 
\label{eqn:sampledProblemComm}
\begin{align} 
&\min_{\substack{\bx, \br, \bp}} \sum_{(i,j) \in \cE} \hspace{-.1cm} h_{ij}^{(k)}(\{r_{ij,s}\}) - \sum_s \sum_{i \in \cN} f_{i,s}^{(k)}(x_{i,s}) + \hspace{-.2cm} \sum_{(i,j) \in \cE} \hspace{-.1cm} v_{ij}^{(k)} (p_{ij}) \label{eq:obj_p0_Comm} \\
&\mathrm{s.\,to:~} x_{i,s} \in \cX_{i,s}^{(k)} , \, \forall \, s = 1,\ldots, S,  i \in \cN \label{eqn:constr_Xsampl_Comm} \\
& \hspace{1.0cm} \sum_s r_{ij,s} + w_{ij}^{(k)} \leq c_{ij}^{(k)}(p_{ij}), \, (i,j) \in \cE \label{eq:ineqcapacityComm} \\
& \hspace{1.1cm} x_{j,s} + y_{j,s,\textrm{in}}^{(k)}(\br) \leq y_{j,\textrm{out},s}^{(k)}(\br) , \, j \in \cN, \forall \, s \label{eq:ineqconstComm} 
 \end{align}
\end{subequations}
where $h_{ij}^{(k)}$ is a given convex function capturing costs associated with communication links; $v_{ij}^{(k)}$ is a convex function associated with the transmission power $p_{ij}$; $f_{i,s}^{(k)}$ is a concave utility function associated with the traffic $x_{i,s}$ generated at node $i$ for flow $s$; $\cX_{i,s}^{(k)} := [0, \bar{x}_{i,s}]$ for a given maximum rate $\bar{x}_{i,s} > 0$ if $i$ generates traffic and $\cX_{i,s}^{(k)} := [-\bar{x}_{i,s}, 0]$ if $i$ is a destination node;
 $c_{ij}^{(k)}(p_{ij})$ is the logarithmic function capturing the capacity of the channel $(i,j) \in \cE$; and, $y_{j,s, \textrm{in}}^{(k)}(\br) $, $y_{j,s,\textrm{out}}^{(k)}(\br)$ are defined as
\begin{subequations} 
\begin{align} 
y_{j,s, \textrm{in}}^{(k)}(\br) & := \sum_{i \in \cN_j^\textrm{in}} (w_{ij}^{(k)} + r_{ij,s}  ) \\
y_{j,s, \textrm{out}}^{(k)}(\br) & := \sum_{i \in \cN_j^\textrm{out}} (  w_{ji}^{(k)}  + r_{ji,s} )
\end{align}
\end{subequations}
respectively. In particular,~\eqref{eq:ineqconstComm} is a relaxed version of the flow-conservation constraint~\cite{Calvo18} for each flow $s$.

To design of the the feedback-based algorithm~\eqref{eq:primalDualFeedback}, the Lagrangian function is built by dualizing the constraint~\eqref{eq:ineqcapacityComm} and~\eqref{eq:ineqconstComm}. In the resultant algorithm, the overall flows $y_{j,\textrm{in}}^{(k)}(\br^{(k)})$ and $y_{j,\textrm{out}}^{(k)}(\br^{(k)})$ entering and exiting a node $i$, and the link flows $\sum_s r_{ij,s} + w_{ij}^{(k)}$ and capacity $c_{ij}^{(k)}$ are replaced by measurements. 
Notice that constraint~\eqref{eq:ineqcapacityComm} is  satisfied strictly during the iterations of the algorithm because of physical limits (communication rates cannot exceed the link capacity). 

\subsection{Example in Power Systems} 
\label{sec:opf}

As an example of application in power systems, consider the problem of optimizing in real time the operation of aggregations of distributed energy resources (DERs) located in (a portion of) a  distribution system (e.g., a distribution feeder) \cite{commelec1,opfPursuit,Tang17}. The problem is formulated as a time-varying optimal power flow problem (OPF) and fits the proposed framework in Section \ref{sec:problemformulation} as described next. Particularly, we consider a distribution  network with one slack bus and $N$ $PQ$-buses. We assume, without loss of generality, that a controllable resource is connected at every $PQ$ bus $i = 1, \ldots, N$. The controllable quantity of resource $i$ is given by $\bx_i := (P_i, Q_i)^\sfT \subseteq \reals^2$, where $P_i$ and $Q_i$ are the net active and reactive power injections from resource $i$, respectively.  
The objective function for a photovoltaic (PV) system $i$ is
$f^{(k)}_i(P, Q)= c_p\left(P- \bar{P}^{(k)}_{i}\right)^2+c_q Q^2$, 
where $\bar{P}^{(k)}_{i}$ is the maximum
real power available at PV system $i$ at time step $k$. We set
\begin{align}
\mathcal{X}_i^{(k)}
=\left\{(P,Q):P^2+Q^2\leq S_{i,\max}^2,\ 
0\leq P \leq P^{(k)}_{i} \right\}
\end{align}
where $S_{i,\max}$ is the rated apparent power for the PV system $i$; similar costs and sets are considered for energy storage systems, with the additional constraints on $P$ based on the current state of charge. 

The function $\by_1^{(k)}(\bx) = \bC_1 \bx + \bD_1 \bw^{(k)} \in \reals^{3 N}$ is the \emph{linearized} mapping from power injections to the voltage magnitudes at each node, derived using, e.g., \cite{Baran89}; the vector $\bw^{(k)}$ collects the uncontrollable power injections at every node at time step $k$. 
One engineering constraint involves the voltage magnitudes at each node to be within an interval $[0.95, 1.05]$ p.u., thus defining the following constraints:
\begin{align}
\bg_1^{(k)}(\by_1^{(k)}(\bx)) & \nonumber \\
& \hspace{-1.0cm} = ([0.95 \textbf{1} - \by_1^{(k)}(\bx)]^\sfT, [\by_1^{(k)}(\bx) - 1.05 \textbf{1}]^\sfT)^\sfT.
\end{align}
Furthermore, we consider a map $\by_2^{(k)}(\bx) = \bC_2 \bx + \bD_2 \bw^{(k)} \in \reals^3$ representing a linear map for the active powers at the three phases at the feeder head~\cite{multiphaseArxiv}; we then impose  a constraint of the form 
\begin{align}
\label{eq:setpoint}
\bg_2^{(k)}(\by_2^{(k)}(\bx)) = \|\by_2^{(k)}(\bx) - \by_{\textrm{ref}}^{(k)}\|_2^2 - \epsilon,
\end{align}
where $\by_{\textrm{ref}}^{(k)}$ is a time-varying reference signal for the active powers at the feeder head and $\epsilon$ is a given accuracy.

In the feedback-based algorithm~\eqref{eq:primalDualFeedback}, measurements of voltages $\hat{\by}_1^{(k)}$ and powers $\hat{\by}_2^{(k)}$ replace the respective network maps $\by_1^{(k)}(\bx)$ and $\by_2^{(k)}(\bx)$, respectively.

\section{Illustrative Numerical Results}
\label{sec:results}

\subsection{Real-Time Routing} 
\label{sec:results_comm_systems}

Consider the network in Fig.~\ref{fig:F_comm_net} with 6 nodes and 8 links (see also~\cite{Chen12}), and assume that two traffic flows are generated by nodes $1$ and $4$, and they are received at nodes $3$ and $6$, respectively. The routing matrix is based on the directed edges.

The cost function for the transmission powers $\bp$ is fixed over time and it is set to $\frac{1}{20}\|\bp\|_2^2$, whereas the utility function associated with the flows are set to $f^{(k)}(x_{i,s}) =  \kappa_{i,s}^{(k)} \log(x_{i,s})$, with $\kappa_{i,s}^{(k)} > 0$ a time-varying coefficient; in particular, to test the algorithm under different conditions, $\kappa_{i,s}^{(k)}$ will be perturbed with Gaussian noise, will follow ramps, step-changes, and sinusoidal signals as will be shown shortly. At each time step, the capacity of the communication channels is generated by using a complex Gaussian random variable with mean $1 + j1$ and variance of $0.01$ for both real and imaginary parts; the exogenous flows $w^{(k)}_{i,j}$ are i.i.d. Gaussian random variables with mean $[0.2, 0.3, 0.3, 0.4, 0.5, 0.2, 0.1, 0.4]^\sfT$ and variance $0.05$ (per entry). In the algorithm, the stepsize if set to $\alpha = 0.5$, and the regularization parameters are $r = 0.001$, and $d = 0.001$. The measurement noise is i.i.d. Gaussian across measurements, with zero mean and variance $0.01$.

\begin{figure}[t!]
  \centering
  \includegraphics[width=.7\columnwidth]{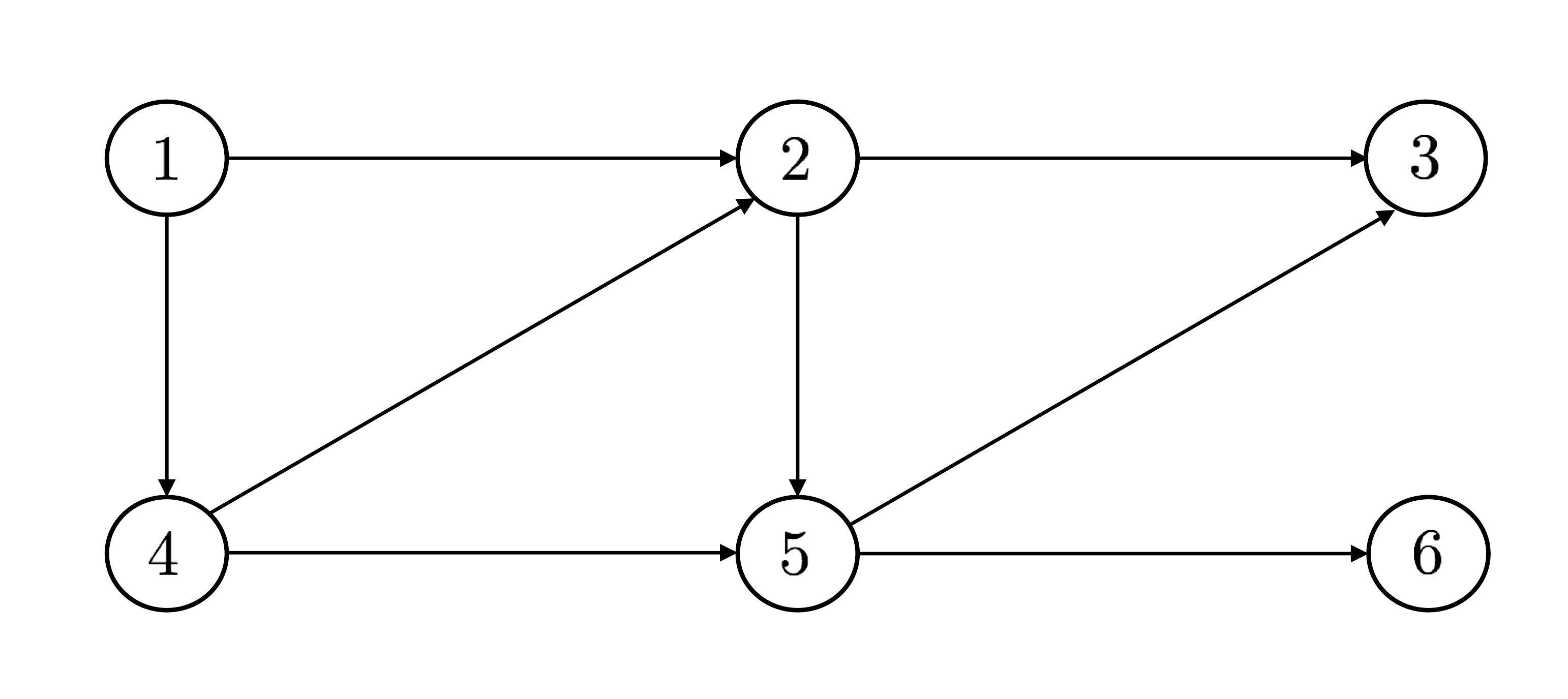}
\vspace{-.4cm}
\caption{Communication network utilized in the numerical results}
\label{fig:F_comm_net}
\vspace{-.2cm}
\end{figure}

\begin{figure}[t!]
  \centering
  \includegraphics[width=1.0\columnwidth]{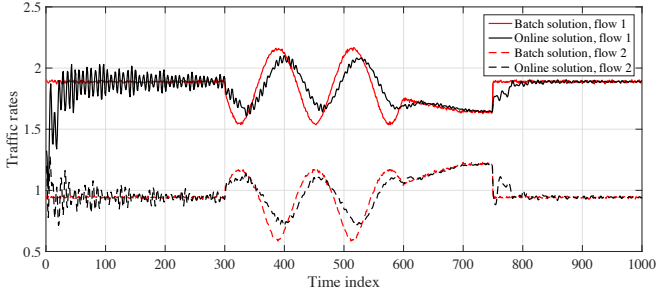}
\vspace{-.6cm}
\caption{Traffic rates: comparison between batch and online solution.} \label{fig:F_traffic_comm}
\vspace{-.2cm}
\end{figure}
\begin{figure}[t!]
  \centering
  \includegraphics[width=1.0\columnwidth]{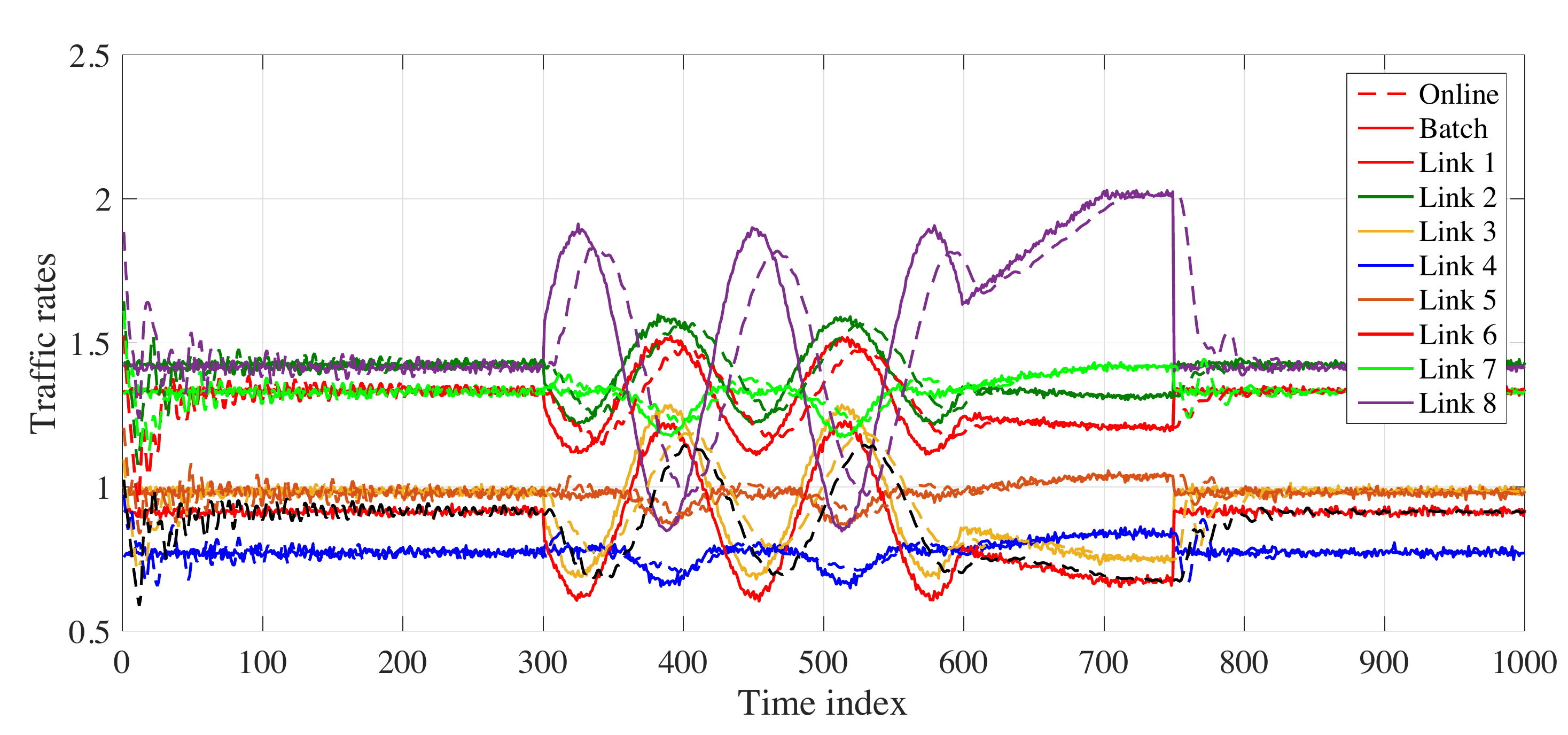}
\vspace{-.6cm}
\caption{Transmit powers: comparison between batch and online solution.} \label{fig:F_powers_comm}
\vspace{-.2cm}
\end{figure}

\begin{figure}[t!]
  \centering
  \includegraphics[width=1.0\columnwidth]{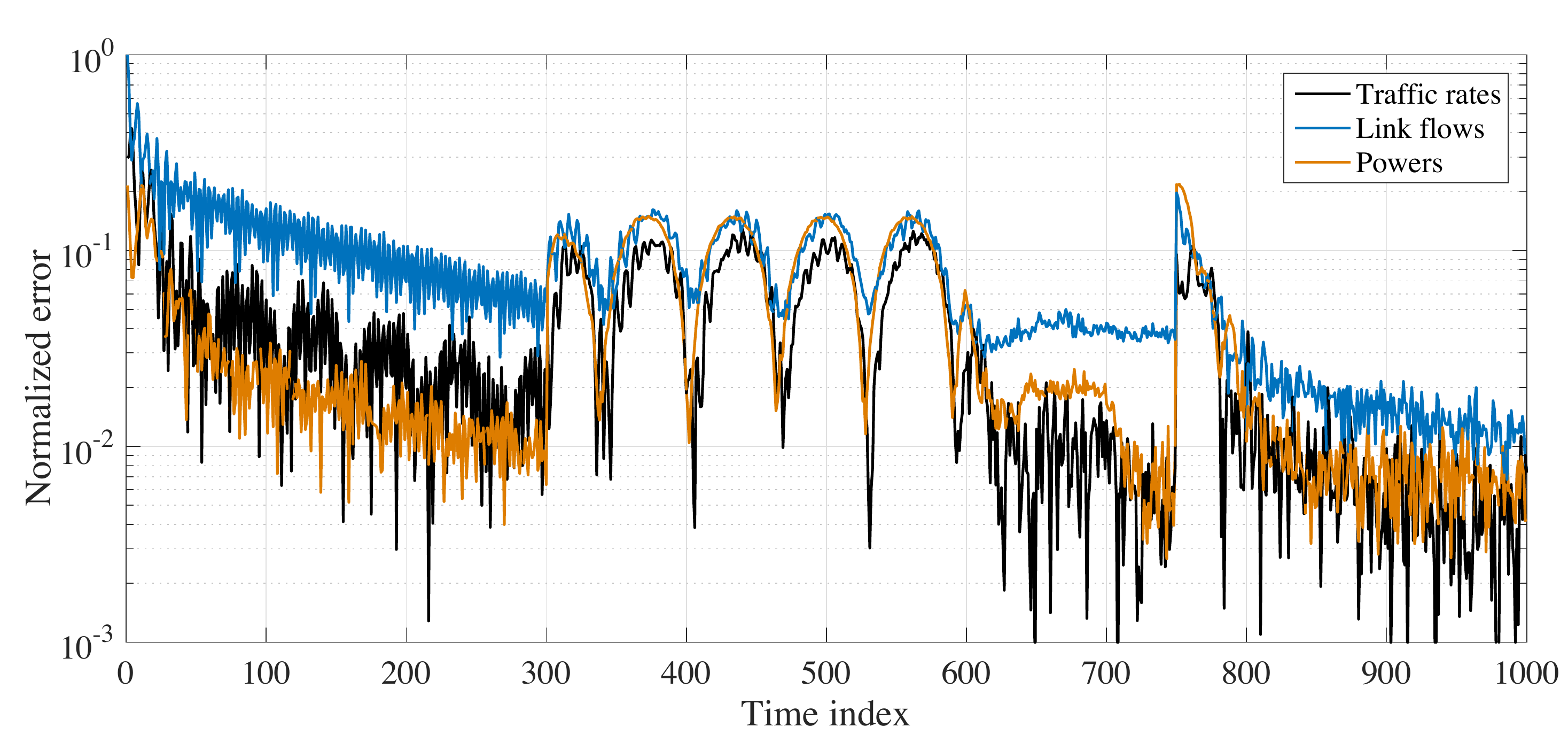}
\vspace{-.6cm}
\caption{Normalized error for traffic rates and link flows for the online algorithm.} \label{fig:F_error_tracking}
\vspace{-.2cm}
\end{figure}

\begin{figure}[t!]
  \centering
  \includegraphics[width=1.0\columnwidth]{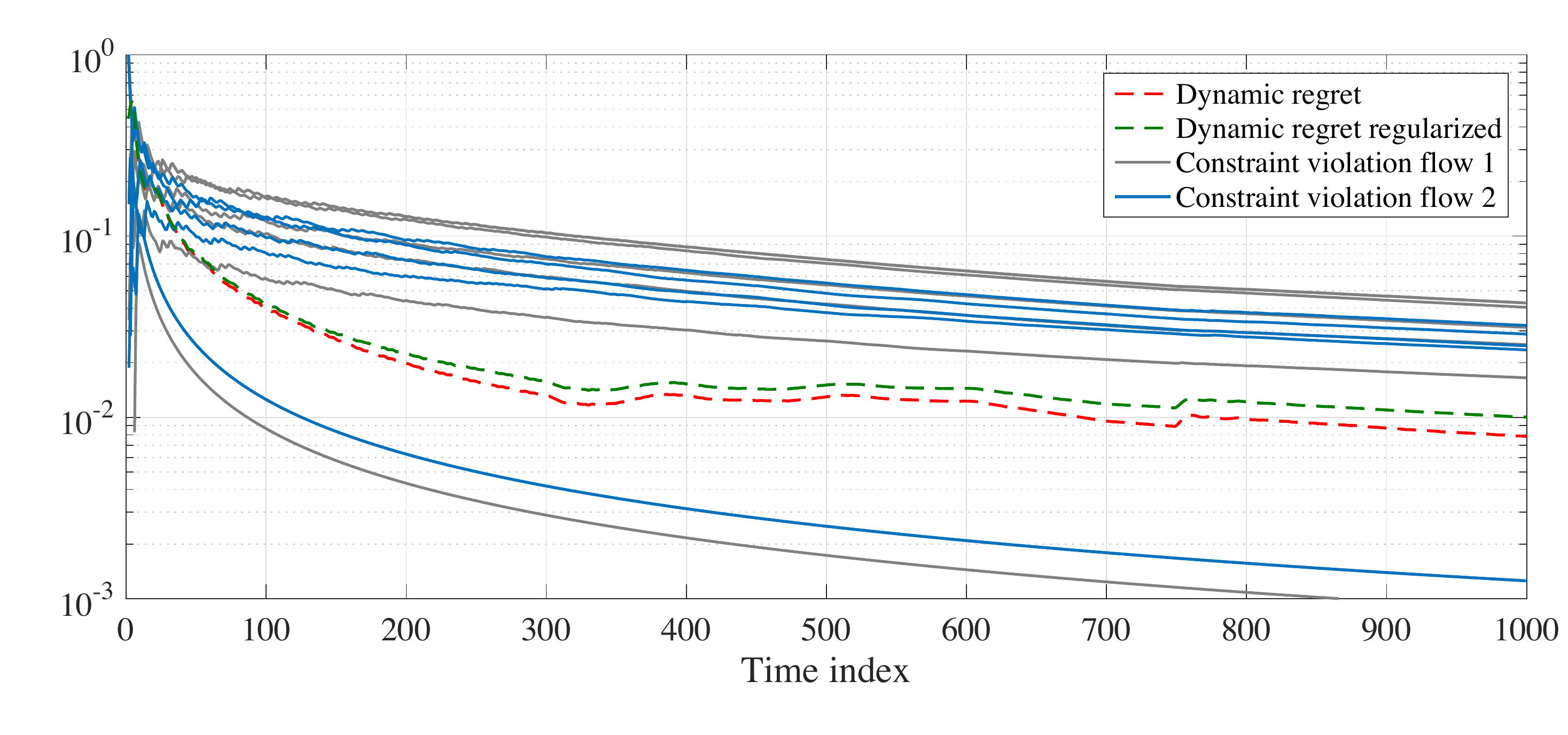}
\vspace{-.6cm}
\caption{Dynamic regret and time-averaged constraint violation.} \label{fig:F_error_comm}
\vspace{-.2cm}
\end{figure}

Fig.~\ref{fig:F_traffic_comm} illustrates the trajectories of the optimal variables $\{x_{i,s}^{(*,k)}\}$ for the traffic rates, obtained by solving problem~\eqref{eqn:sampledProblemComm} to convergence at each time step $k$ (black trajectories); these trajectories are based on the variability of cost, communication channel gain, and exogenous traffic, and they are taken as a benchmark. As an effect of the time-varying costs and problem inputs, the optimal trajectories feature a mix of small disturbances, continuously time-varying portions, and step changes.  Fig.~\ref{fig:F_traffic_comm} also illustrates the trajectories of the traffic rates produced by the online algorithm. It can be seen that the variables $\{x_{i,s}^{(k)}\}$ closely follow the optimal points. A similar trend is observed in Fig.~\ref{fig:F_powers_comm}, where the trajectories of the transmit powers are illustrated. 

The normalized tracking error $\frac{\|\bx^{(*,k)}-\bx^{(k)}\|_2}{\|\bx^{(*,k)}\|_2}$ is illustrated in Fig.~\ref{fig:F_error_tracking}, along with the normalized error for link flows $\frac{\|\br^{(*,k)}-\br^{(k)}\|_2}{\|\br^{(*,k)}\|_2}$ and powers $\frac{\|\bp^{(*,k)}-\bp^{(k)}\|_2}{\|\bp^{(*,k)}\|_2}$. Similar trends can be observed, with a momentary increase of the error around iterations $500$ and $750$ where the optimal solution follow a sinusoidal trajectory and experience a step change. The error momentarily increases when $\sigma^{(k)}$ is larger, thus corroborating the analytical findings.

Lastly, Fig.~\ref{fig:F_error_comm} illustrates the time-averaged constraint violation for the flow conservation constraints and the dynamic regret. In particular, the dynamic regret is plotted for both Case 1 and Case 2. it can be seen that in Case 2 (i.e., regularized Lagrangian) the dynamic regret is slightly higher; this is because the algorithm tracks approximate KKT points.

\subsection{Real-Time OPF} 
\label{sec:results_power_flow}

An illustrative numerical result for the AC OPF problem is provided here based on the test case  described in~\cite{opfPursuit}; in particular,~\cite{opfPursuit} considered the IEEE 37-node test feeder, the distribution system was populated with photovoltaic (PV) systems and energy storage systems, and real datasets for non-controllable loads and solar irradiance were utilized with a granularity of $1$ second; see~\cite{opfPursuit} for a detailed description of the dataset and the simulation setup. 

\begin{figure}[t!]
  \centering
  \includegraphics[width=.95\columnwidth]{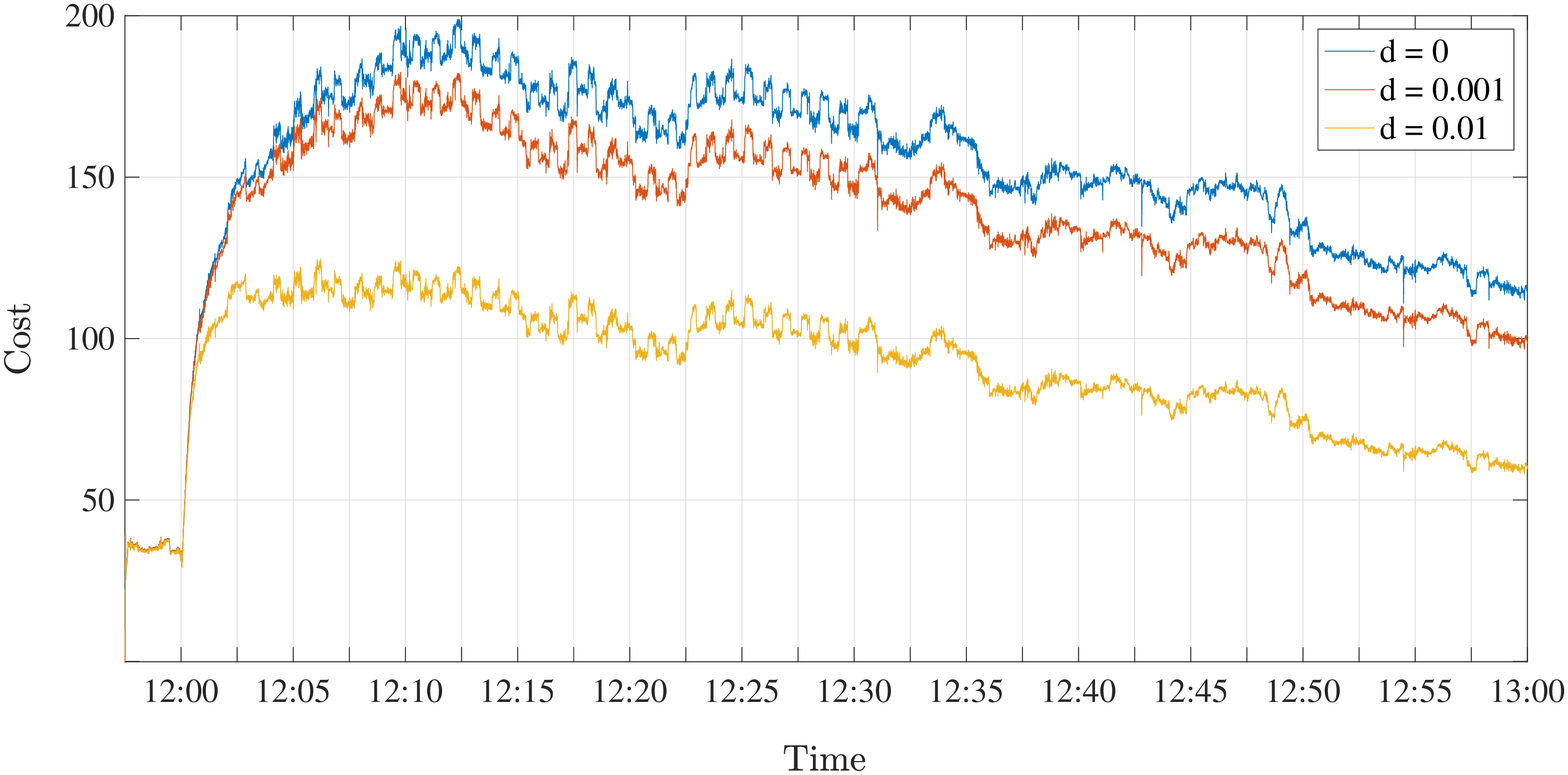}
\vspace{-.25cm}
\caption{Achieved cost of the real-time OPF for different values of the parameter $d$.} \label{fig:F_cost_norm}
\vspace{-.2cm}
\end{figure}
\begin{figure}[t!]
  \centering
  \includegraphics[width=.95\columnwidth]{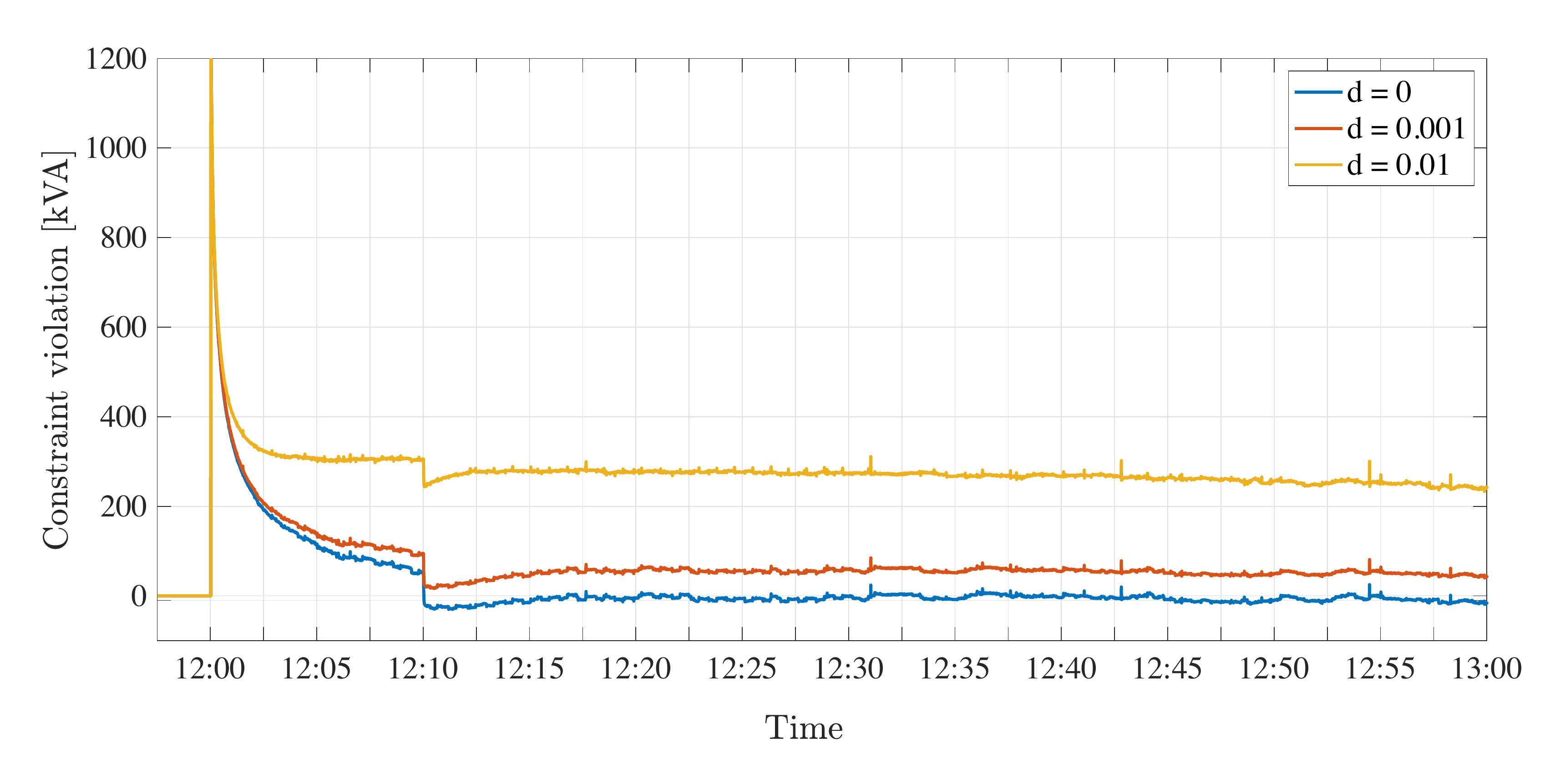}
\vspace{-.25cm}
\caption{Constraint violation for $\|\by_2^{(k)}(\bx) - \by_{\textrm{ref}}^{(k)}\|_2^2 \leq \epsilon$ of the real-time OPF for different values of the parameter $d$.} \label{fig:F_violation_norm}
\vspace{-.2cm}
\end{figure}

We consider a map $\by_2^{(k)}(\bx) = \bC_2 \bx + \bD_2 \bw^{(k)}$ representing a linear approximation for the active and reactive powers at the feeder head~\cite{multiphaseArxiv} and we impose a constraint~\eqref{eq:setpoint}
where $\by_{\textrm{ref}}^{(k)} = [-2500 \textrm{~kW}, -1250 \textrm{~kVAr}]^\sfT$ from 12:00pm to 12:10pm, and then $[-2400 \textrm{~kW}, -1200 \textrm{~kVAr}]^\sfT$ from 12:10pm to 13:00pm, and $\epsilon = 20$ [kVA$^2$]. The constraint~\eqref{eq:setpoint} is not imposed before 12:00pm and after 13:00pm. 

The cost  in the OPF and the voltage constraints are set as in~\cite{opfPursuit}. Figures~\ref{fig:F_cost_norm} and~\ref{fig:F_violation_norm} illustrate the cost achieved by the real-time OPF algorithm and the constraint violation for $\|\by_2^{(k)}(\bx) - \by_{\textrm{ref}}^{(k)}\|_2^2 \leq \epsilon$, respectively, for different values of the regularization coefficient $d$. Since the constraint is not imposed before 12:00pm, a transient is experienced during the first couple of minutes. As described in the previous sections, it can be seen that positive values of $d$ lead to tracking of approximate KKT points and, therefore, to a systematic constraint violation. Since the constraint is violated, a smaller cost can be achieved. Notice that, even though we consider a linearized map, the nonlinearity of the AC power flow equations is implicitly taken into account by   the algorithm through the feedback; therefore, for $d = 0$, the algorithm still guarantees satisfaction of the constraint.

\section{Conclusion}
\label{sec:Conclusions}

This paper leveraged a time-varying convex optimization formalism to model optimal operational trajectories of systems or network of systems, and developed feedback-based online algorithms based on primal-dual projected-gradient methods. In the proposed algorithms, the gradient steps were  modified to accommodate  measurements from the network system. When the design of the algorithm is based on the time-varying Lagrangian, the paper characterized the performance of the proposed via a dynamic regret analysis. When a regularized Lagrangian is utilized, results in terms of Q-linear convergence are provided, at the cost of tracking an approximate KKT trajectory.      

Extending the proposed methodology to time-varying nonconvex problems is the subject of current research efforts. Future efforts will also look at characterizing the performance of the propose method when implemented in a distributed and asynchronous fashion.

\appendix

\subsection{Proofs of Theorem~\ref{thm:regret} and Theorem~\ref{thm:violation}}
\label{sec:proof_regret}

To prove Theorem~\ref{thm:regret} and Theorem~\ref{thm:violation},  the following intermediate  results are first shown.

\begin{lemma} \label{lem:Lx}
For any $\blambda \in \Lambda_{\alpha, \kappa}$, the following holds: 
\begin{align}
&\left \|\nabla_{\bx} \cL^{(k)}\left (\bx^{(k)}, \blambda\right) - \hat{\nabla}_{\bx} \cL^{(k)}\left (\bx^{(k)}, \blambda\right) \right \|_2 \leq  L_x e_y.
\end{align}
Moreover, $\|\nabla_{\bx} \cL^{(k)}\left (\bx, \blambda\right)\|$ is uniformly bounded by $F_x$.
\end{lemma}

\begin{proof}
Note that
\begin{align}
  \nabla_x \cL^{(k)} (\bx, \blambda) &:= \nabla f^{(k)} (\bx) + \bC^\sfT \nabla f_0^{(k)} (\by^{(k)}(\bx)) \nonumber \\
 & \qquad + \left( \bJ^{(k)} (\by^{(k)} (\bx)) \bC \right)^\sfT \blambda
 \end{align}
By comparing with \eqref{eq:gradXApprox}, we have that
\begin{align*}
&\left \|\nabla_{\bx} \cL^{(k)}\left (\bx^{(k)}, \blambda\right) - \hat{\nabla}_{\bx} \cL^{(k)}\left (\bx^{(k)}, \blambda\right) \right \|_2 \\
& \quad \leq \|\bC\|_2 \|\nabla f_0^{(k)} (\by^{(k)} (\bx^{(k)})) - \nabla f_0^{(k)}(\hat{\by}^{(k)}) \|_2 \\
& \qquad + \|\blambda\|_2 \|\bC\|_2 \|\bJ^{(k)} (\by^{(k)}(\bx^{(k)})) - \bJ^{(k)} (\hat{\by}^{(k)})  \| \\
& \quad \leq \|\bC\|_2 L_0  \|\by^{(k)} (\bx^{(k)}) - \hat{\by}^{(k)} \|_2 \\
& \qquad + \|\blambda\|_2 \|\bC\|_2 L_G \|\by^{(k)} (\bx^{(k)}) - \hat{\by}^{(k)} \|_2 \\
& \quad \leq \left(\|\bC\|_2 L_0 + \frac{1}{\alpha^\kappa} \|\bC\|_2  L_G\right) e_y,
\end{align*}
where the first inequality holds by the triangle and Cauchy-Schwarz inequalities; the second inequality follows by Assumptions \ref{ass:cost} and \ref{ass:nonlinearconstr}; and the last inequality is due to the fact that $\blambda \in \Lambda_{\alpha, \kappa}$ and Assumption \ref{ass:error}.
\end{proof}

\begin{lemma}\label{lem:Llam}
For any $\blambda$, it  holds that: 
\[
\left \| \bg^{(k)}(\hat{\by}^{(k)} )  - \nabla_{\lambda} \cL^{(k)}\left (\bx^{(k)}, \blambda\right) \right \|_2 \leq M_g e_y.
\]
Furthermore, $\|\nabla_{\lambda} \cL^{(k)}\left (\bx, \blambda\right)\|$ is uniformly bounded by $g$.
\end{lemma}

\begin{proof}
The result follows from \eqref{eq:lip_g} and Assumption \ref{ass:error}.
\end{proof}

\begin{lemma} \label{lem:inst_bound}
For every $k$, the following inequality holds:
  \begin{align*}
  &\cL^{(k)}\left (\bx^{(k)}, \blambda^{(k)}\right)  - \cL^{(k)}\left (\bx^{(*, k)}, \blambda^{(k)}\right)  \\
  &\leq [ \|\bx^{(k)} - \bx^{(*, k)}\|^2 - \|\bx^{(k+1)} - \bx^{(*, k+1)}\|^2 ]/(2 \alpha) + \alpha F_x^2/2 \\
  &\quad +  \left[(2B +  \alpha F_x) +  \frac{\alpha}{2} L_x e_y\right]  L_x e_y + \frac{D + B}{\alpha} \|\bx^{*, k} -  \bx^{(*, k+1)}\| .
  \end{align*}
Furthermore, for any $\blambda \in \Lambda_{\alpha, \kappa}$, it holds that:
  \begin{align*}
  &\cL^{(k)}\left (\bx^{(k)}, \blambda^{(k)}\right)  - \cL^{(k)}\left (\bx^{(k)}, \blambda\right)  \\
  &\geq [ \|\blambda^{(k+1)} - \blambda\|^2 - \|\blambda^{(k)} - \blambda\|^2 ]/(2 \alpha) - \alpha g^2/2 \\
  &\quad -  \left[\frac{2}{\alpha^\kappa} + \alpha \left(g + \frac{M_g e_y}{2}\right) \right]M_g e_y \, .
  \end{align*}
 
  \end{lemma}
  
  \begin{proof}
  We have that
  \begin{align}
    &\|\bx^{(k+1)} - \bx^{(*, k+1)}\|^2\nonumber \\
    &= \|\bx^{(k+1)} - \bx^{(*, k)} +  \bx^{(*, k)}-  \bx^{(*, k+1)}\|^2 \nonumber \\
    &= \|\bx^{(k+1)} - \bx^{(*, k)}\|^2 + 2 (\bx^{(k+1)} - \bx^{(*, k)})^\sfT (\bx^{(*, k)}-  \bx^{(*, k+1)}) \nonumber \\
    &\quad + \|\bx^{(*, k)}-  \bx^{(*, k+1)}\|^2 \nonumber \\
    &\leq \left \|\bx^{(k)} - \bx^{(*, k)} - \alpha \hat{\nabla}_{\bx} \cL^{(k)}\left (\bx^{(k)}, \blambda^{(k)}\right) \right \|^2 \nonumber \\
    &\quad + [2 (\bx^{(k+1)} - \bx^{(*, k)}) + (\bx^{(*, k)}-  \bx^{(*, k+1)})]^\sfT (\bx^{(*, k)}-  \bx^{(*, k+1)}) \nonumber\\
 &\leq \left\|\bx^{(k)} - \bx^{(*, k)} - \alpha \hat{\nabla}_{\bx} \cL^{(k)}\left (\bx^{(k)}, \blambda^{(k)}\right) \right\|^2  \nonumber\\
 &\quad + 2(D + B) \|\bx^{(*, k)}-  \bx^{(*, k+1)}\| \label{eqn:proof_step1}, 
  \end{align}
  where the first inequality follows by \eqref{eq:primalstep_regret} and the non-expansiveness property of the projection operator; and in the last inequality, we used the Cauchy-Schwarz inequality and the fact that under Assumption \ref{ass:sets}
  \begin{align*}
  & \|2 (\bx^{(k+1)} - \bx^{(*, k)}) + (\bx^{(*, k)}-  \bx^{(*, k+1)})\| \\
  &\quad \leq 2 \|\bx^{(k+1)} - \bx^{(*, k)}\| + \|\bx^{(*, k)}\| + \|\bx^{(*, k+1)}\| \\ 
  &\quad \leq 2 \diam (\cY^{(k)}) + |\cY^{(k)}| + |\cY^{(k+1)}|\\
  &\quad \leq 2 (D + B).
  \end{align*}
  We now expand the first term in \eqref{eqn:proof_step1}. It holds that
\begin{align}
& \left\|\bx^{(k)} - \bx^{(*, k)} - \alpha \hat{\nabla}_{\bx} \cL^{(k)}\left (\bx^{(k)}, \blambda^{(k)}\right) \right\|^2  \nonumber \\
& = \Big \| \bx^{(k)}- \bx^{(*, k)} - \alpha \nabla_{\bx} \cL^{(k)}\left (\bx^{(k)}, \blambda^{(k)}\right)  \nonumber \\
& \quad + \alpha \left (\nabla_{\bx} \cL^{(k)}\left (\bx^{(k)}, \blambda^{(k)}\right) - \hat{\nabla}_{\bx} \cL^{(k)}\left (\bx^{(k)}, \blambda^{(k)}\right) \right) \Big \|^2 \label{eqn:proof_step2} 
\end{align}
Let
\begin{align*}
&\bgamma^{(k)} := \alpha\left (\nabla_{\bx} \cL^{(k)}\left (\bx^{(k)}, \blambda^{(k)}\right) - \hat{\nabla}_{\bx} \cL^{(k)}\left (\bx^{(k)}, \blambda^{(k)}\right) \right)
\end{align*}
and note that using Lemma \ref{lem:Lx}, we have
\begin{equation}
\|\bgamma^{(k)}\| \leq \alpha L_x e_y. \label{eq:gamma_k}
\end{equation}
Continuing the derivation in \eqref{eqn:proof_step2}, we obtain 
\begin{align}
& \left\|\bx^{(k)} - \bx^{(*, k)} - \alpha \hat{\nabla}_{\bx} \cL^{(k)}\left (\bx^{(k)}, \blambda^{(k)}\right) \right\|^2  \nonumber \\
    &= \left\|\bx^{(k)}- \bx^{(*, k)} - \alpha \nabla_{\bx} \cL^{(k)}\left (\bx^{(k)}, \blambda^{(k)}\right) \right\|^2 \nonumber \\
    &\quad + \left[2\left(\bx^{(k)}- \bx^{(*, k)} - \alpha \nabla_{\bx} \cL^{(k)}\left (\bx^{(k)}, \blambda^{(k)}\right)\right) + \bgamma^{(k)}  \right]^\sfT \bgamma^{(k)} \nonumber\\
&\leq \left\|\bx^{(k)}- \bx^{(*, k)} - \alpha \nabla_{\bx} \cL^{(k)}\left (\bx^{(k)}, \blambda^{(k)}\right) \right\|^2 \nonumber\\
&\quad + [2(2B +  \alpha F_x) +  \alpha L_x e_y]  \alpha L_x e_y  \nonumber\\
    &= \left\|\bx^{(k)}- \bx^{(*, k)}\right \|^2 \nonumber\\
&\quad - 2 \alpha \left[\nabla_{\bx} \cL^{(k)}\left (\bx^{(k)}, \blambda^{(k)}\right) \right]^\sfT (\bx^{(k)}- \bx^{(*, k)}) \nonumber \\
    &\quad + \alpha^2 F_x^2 + [2(2B +  \alpha F_x) +  \alpha L_x e_y]  \alpha L_x e_y \nonumber\\
    &\leq \left\|\bx^{(k)}- \bx^{(*, k)}\right \|^2 \nonumber\\
    &\quad - 2 \alpha \left[\cL^{(k)}\left (\bx^{(k)}, \blambda^{(k)}\right)  - \cL^{(k)}\left (\bx^{(*, k)}, \blambda^{(k)}\right) \right]  \nonumber\\
  &\quad + \alpha^2 F_x^2 + [2(2B +  \alpha F_x) +  \alpha L_x e_y]  \alpha L_x e_y,  \label{eqn:proof_step3}
  \end{align}
  where  the first inequality holds by the Cauchy-Schwarz inequality, \eqref{eq:gamma_k}, and  Assumption \ref{ass:sets}; and the last inequality holds by the convexity of $\cL^{(k)}\left (\cdot, \blambda\right)$. The first part of the lemma then follows by combining \eqref{eqn:proof_step1} and \eqref{eqn:proof_step3}, and rearranging.
  
  For the second part, for any $\blambda \in \Lambda_{\alpha, \kappa}$, we have that
  \begin{align}
    &\|\blambda^{(k+1)} - \blambda\|^2\nonumber \\
    &\leq \left \|\blambda^{(k)}  -  \blambda + \alpha \bg^{(k)}(\hat{\by}^{(k)} )   \right \|^2 \nonumber \\
    & =  \Big \|\blambda^{(k)}  -  \blambda + \alpha \nabla_{\lambda} \cL^{(k)}\left (\bx^{(k)}, \blambda^{(k)}\right) \nonumber \\
    & \quad + \alpha \left(\bg^{(k)}(\hat{\by}^{(k)} ) - \nabla_{\lambda} \cL^{(k)}\left (\bx^{(k)}, \blambda^{(k)}\right) \right)  \Big \|^2. \nonumber
    \end{align}
    Letting
    \[
    \bbeta^{(k)} := \alpha \left(\bg^{(k)}(\hat{\by}^{(k)} )  - \nabla_{\lambda} \cL^{(k)}\left (\bx^{(k)}, \blambda^{(k)}\right) \right)
    \]
    and noticing that, by Lemma \ref{lem:Llam}
    \begin{align}
 \|\bbeta^{(k)}\|_2 &\leq  \alpha M_g e_y\label{eqn:proof_step11}, 
  \end{align}
  it follows that
  \begin{align}
  &\|\blambda^{(k+1)} - \blambda\|^2\nonumber \\
  &\leq \left \| \blambda^{(k)}  -  \blambda + \alpha \nabla_{\lambda} \cL^{(k)}\left (\bx^{(k)}, \blambda^{(k)}\right)\right \|^2 \nonumber \\
  &\quad + \left[2\left( \blambda^{(k)}  -  \blambda + \alpha \nabla_{\lambda} \cL^{(k)}\left (\bx^{(k)}, \blambda^{(k)}\right)\right) +\bbeta^{(k)}\right]^\sfT \bbeta^{(k)} \nonumber \\
  &\leq \left \| \blambda^{(k)}  -  \blambda + \alpha \nabla_{\lambda} \cL^{(k)}\left (\bx^{(k)}, \blambda^{(k)}\right)\right \|^2 \nonumber \\
  &\quad + \alpha \left[\frac{4}{\alpha^\kappa} + \alpha (2g + M_g e_y) \right]M_g e_y \nonumber \\
  &\leq \left \| \blambda^{(k)}  -  \blambda\right\|^2  + 2\alpha  \left[\nabla_{\lambda} \cL^{(k)}\left (\bx^{(k)}, \blambda^{(k)}\right) \right]^\sfT (\blambda^{(k)}  -  \blambda) \nonumber \\
  &\quad + \alpha^2 g^2 + \alpha \left[\frac{4}{\alpha^\kappa} + \alpha (2g + M_g e_y) \right]M_g e_y \nonumber \\
  &= \left \| \blambda^{(k)}  -  \blambda\right\|^2  + 2\alpha  \left[ \cL^{(k)}\left (\bx^{(k)}, \blambda^{(k)}\right) -  \cL^{(k)}\left (\bx^{(k)}, \blambda\right) \right] \nonumber \\
  &\quad + \alpha^2 g^2 + \alpha \left[\frac{4}{\alpha^\kappa} + \alpha (2g + M_g e_y) \right]M_g e_y \nonumber
  \end{align}
  where the second inequality holds by \eqref{eqn:proof_step11}; and the equality holds by the linearity of $\cL^{(k)}\left (\bx, \blambda\right)$ in $\blambda$. The second part of the lemma then follow by rearranging the obtained inequality.
  \end{proof}

With these intermediate results in place, the proofs of Theorem~\ref{thm:regret} and Theorem~\ref{thm:violation} are  provided next.  

\begin{proof}[Proof of Theorem \ref{thm:regret}]
By using Lemma \ref{lem:inst_bound}, we have that
\begin{align}
& \cL^{(k)}\left (\bx^{(k)}, \blambda\right)  - \cL^{(k)}\left (\bx^{(*, k)}, \blambda^{(k)}\right) \nonumber\\
  &\leq [ \|\bx^{(k)} - \bx^{(*, k)}\|^2 - \|\bx^{(k+1)} - \bx^{(*, k+1)}\|^2 ]/(2 \alpha) + \alpha F_x^2/2 \nonumber\\
  &\quad +  \left[(2B +  \alpha F_x) +  \frac{\alpha}{2} L_x e_y\right]  L_x e_y \nonumber \\
  & \quad +  (D + B) \|\bx^{*, k} -  \bx^{*, k+1}\|/\alpha \nonumber \\
    &\quad  + [ \|\blambda^{(k)} - \blambda\|^2 - \|\blambda^{(k+1)} - \blambda\|^2 ]/(2 \alpha) + \alpha g^2/2 \nonumber \\
  &\quad + \left[\frac{2}{\alpha^\kappa} + \alpha \left(g + \frac{M_g e_y}{2}\right) \right]M_g e_y\label{eqn:bound_reg_inst}
\end{align}	
for any $\blambda \in \Lambda_{\alpha, \kappa}$. To show \eqref{eqn:regret}, we use $\blambda = 0$ and the fact that $[\blambda^{(k)}]^\sfT \bg^{(k)}(\by^{(k)}(\bx^{(*, k)})) \leq 0$ by the feasibility of $\bx^{(*, k)}$ for $\textrm{(P0)}^{(k)}$. Therefore, by \eqref{eqn:lagrangian}, we have that
\begin{align}
&\cL^{(k)}\left (\bx^{(k)}, 0\right)  - \cL^{(k)}\left (\bx^{(*, k)}, \blambda^{(k)}\right) \nonumber \\
&= h^{(k)} (\bx^{(k)}) - h^{(k)} (\bx^{(*,k)}) - [\blambda^{(k)}]^\sfT \bg^{(k)}(\by^{(k)}(\bx^{(*, k)})) \nonumber \\
&\geq h^{(k)} (\bx^{(k)}) - h^{(k)} (\bx^{(*,k)}).
\end{align}
By using this last inequality in \eqref{eqn:bound_reg_inst}, and summing  \eqref{eqn:bound_reg_inst} over $\ell = 1, \ldots, k$, we have
\begin{align}
& \frac{1}{k}\sum_{\ell = 1}^k \left(h^{(k)} (\bx^{(k)}) - h^{(k)} (\bx^{(*,k)})\right ) \nonumber\\
& \leq \frac{1}{2 \alpha k} \left(\|\bx^{(1)} - \bx^{(*, 1)}\|^2 + \|\blambda^{(1)} \|^2 \right) + \frac{\alpha}{2} ( F_x^2 +  g^2) \nonumber\\
& + \left[(2B +  \alpha F_x) +  \frac{\alpha}{2} L_x e_y\right]  L_x e_y \nonumber \\
& + \left[\frac{2}{\alpha^\kappa} + \alpha \left(g + \frac{M_g e_y}{2}\right) \right]M_g e_y  \nonumber\\
& + \frac{D + B}{\alpha} \frac{1}{k}\sum_{\ell=1}^k \sigma^{(\ell)}, \label{eqn:regret_proof}
\end{align}
Note that $F_x = F + G/\alpha^\kappa$, and hence
\begin{align*}
\alpha F_x^2 &= \alpha (F^2 + 2 F G/\alpha^\kappa + G^2/\alpha^{2 \kappa}) \\
& = \alpha F^2 + 2 F G \alpha^{1- \kappa} + G^2 \alpha^{1 - 2 \kappa}.
\end{align*}
Using this, the fact that $\|\bx^{(1)} - \bx^{(*, 1)} \| \leq 2B $, and assuming (without loss of generality) that $\blambda^{(1)} = 0$, completes the proof of \eqref{eqn:regret_anytime} and
\eqref{eqn:regret}.
\end{proof}

\begin{proof}[Proof of Theorem \ref{thm:violation}]
To prove \eqref{eqn:constraints_anytime}, for a given $j = 1, \ldots, M$, consider the $j$-th component of $\bg^{(k)}(\by^{(k)}(\bx^{(k)}))$, and let $\blambda_j$ be a vector in $\reals^M_{+}$ with all zero components apart from the $j$-the component which equals $1/\alpha^\kappa$. Note that $\blambda_j \in \Lambda_{\alpha, \kappa}$ by construction, and $\blambda_j^\sfT \bg^{(k)}(\by^{(k)}(\bx^{(k)})) = g^{(k)}_j(\by^{(k)}(\bx^{(k)}))/\alpha^\kappa$. Therefore, 
\begin{align*}
&\cL^{(k)}\left (\bx^{(k)}, \blambda_i\right)  - \cL^{(k)}\left (\bx^{(*, k)}, \blambda^{(k)}\right) \\
&\quad \geq h^{(k)} (\bx^{(k)}) - h^{(k)} (\bx^{(*,k)}) + g_j^{(k)}(\by^{(k)}(\bx^{(k)}))/\alpha^\kappa.
\end{align*}
Now, by convexity of $h$ and Assumption	\ref{ass:sets}, we have that
\begin{align}
h^{(k)} (\bx^{(*,k)}) - h^{(k)} (\bx^{(k)}) &\leq \left(\nabla h^{(k)} (\bx^{(*,k)}) \right)^\sfT \left(\bx^{(*,k)} -  \bx^{(k)} \right) \nonumber \\
& \leq \left \| \nabla h^{(k)} (\bx^{(*,k)}) \right \| \left\| \bx^{(*,k)} -  \bx^{(k)}\right\| \nonumber \\
&\leq 2FB
\end{align}
Thus, letting $B^{(\infty)}(\alpha, \kappa)$ denote the asymptotic bound of \eqref{eqn:regret_proof}, and using similar derivation, completes the proof of the theorem.
\end{proof}

\subsection{Proofs of Theorem~\ref{theorem.inexact} and Theorem~\ref{cor:asym}}
\label{sec:proof_linear}

The following intermediate lemmas are utilized for the proof of~Theorem~\ref{theorem.inexact} and Theorem~\ref{cor:asym}.    

\begin{lemma}
\label{lem:errorPhi}
Assumptions \ref{ass:sets}, \ref{ass:cost}, \ref{ass:nonlinearconstr}, and \ref{ass:error}, the perturbation in the map $\hat{\bphi}^{(k)}$ can be bounded as: 
\begin{align}
\label{eq:errorPhi}
\left\|\bphi^{(k)}(\bz^{(k)}) - \hat{\bphi}^{(k)} (\bz^{(k)}) \right\|_2^2 \leq e_p^2 + e_d^2
\end{align}
where 
\begin{align}
e_p & \leq \left(L_0 + M_\lambda M_I \max_{j = 1, \ldots, M_I}\{L_{g_j}\}\right) \|\bC\|_2 e_y \\
e_d & \leq M_g e_y \, ,
\end{align}
with $M_{\lambda} := \sup_{k \geq 1} \max_{\blambda \in \cD^{(k)}} \|\blambda\|_1$\, . 
\end{lemma}
\begin{proof}
Notice first that the left hand side of~\eqref{eq:errorPhi} can be written as $\|\be_p^{(k)}\|^2 + \|\be_d^{(k)}\|^2$, where  
\begin{align}
\be_p^{(k)} & := \nabla_{\bx}  \cL_{p,d}^{(k)}(\bz^{(k)}) - \hat{\nabla}_{\bx}  \cL_{p,d}^{(k)}(\bz^{(k)}) \label{eq_error_frad_p} \\
\be_d^{(k)} & := \bg^{(k)}(\by^{(k)}(\bx^{(k)})) - \bg^{(k)}(\hat{\by}^{(k)}) \, . \label{eq_error_frad_d}
\end{align}
Regarding~\eqref{eq_error_frad_d}, from~\eqref{eq:lip_g} and Assumption~\ref{ass:error}, it follows that  
\begin{align} 
\|\be_d^{(k)}\|_2 \leq M_g \|\by^{(k)}(\bx^{(k)}) - \hat{\by}^{(k)}\|_2 \leq M_g e_y
\end{align}   
for all $k \in \mathbb{N}$. Regarding~\eqref{eq_error_frad_p}, use the triangle inequality to obtain $\|\be_p^{(k)}\|_2 \leq \|\be_{p,1}^{(k)}\|_2 + \|\be_{p,2}^{(k)}\|_2$, with: 
\begin{align}
& \be_{p,1}^{(k)} := \bC^\sfT \nabla f_0^{(k)} (\by^{(k)}(\bx^{(k)})) - \bC^\sfT \nabla f_0^{(k)} (\hat{\by}^{(k)})  \\
& \be_{p,2}^{(k)} := \sum_{j = 1}^M \lambda^{(k)}_j \bC^\sfT \left(\nabla_\bx g_j(\by^{(k)}(\bx^{(k)})) - \nabla_\bx g_j(\hat{\by}^{(k)}) \right) .
\end{align}
The first term can be bounded as
\begin{subequations}
\begin{align}
\|\be_{p,1}^{(k)}\|_2 & \leq \|\bC\|_2 \| \nabla f_0^{(k)} (\by^{(k)}(\bx^{(k)})) - \nabla f_0^{(k)} (\hat{\by}^{(k)}) \|_2  \\
& \leq \|\bC\|_2 L_0 \|\by^{(k)}(\bx^{(k)}) - \hat{\by}^{(k)}\|_2 \\
& \leq \|\bC\|_2 L_0 e_y \, .
\end{align}
\end{subequations}
The norm of $\be_{p,2}^{(k)}$ can be bounded as follows:
\begin{subequations}
\begin{align}
& \|\be_{p,2}^{(k)}\|_2 \nonumber \\
& \leq 
\sum_{j = 1}^M |\lambda_j^{(k)}| \|\bC^\sfT (\nabla_\bx g_j(\by^{(k)}(\bx^{(k)})) - \nabla_\bx g_j(\hat{\by}^{(k)}))\|_2 \\
& \leq \|\bC\|_2 \sum_{j = 1}^M |\lambda_j^{(k)}|\|\nabla_\bx g_j(\by^{(k)}(\bx^{(k)})) - \nabla_\bx g_j(\hat{\by}^{(k)})\|_2 \hspace{-.2cm} \\
& \leq \|\bC\|_2 \sum_{j = 1}^M |\lambda_j^{(k)}| L_{g_j} \|\by^{(k)}(\bx^{(k)}) - \hat{\by}^{(k)}\|_2 \\
& \leq \|\bC\|_2 \sum_{j = 1}^M |\lambda_j^{(k)}| L_{g_j} e_y \\
& \leq \|\bC\|_2 \|\blambda^{(k)}\|_1 M_I \max_{i \ 1, \ldots, M_I} \{L_{g_j}\} e_y \, .
\end{align}
\end{subequations}
Using the definition of $M_{\lambda}$, the result follows.
\end{proof}

\begin{lemma}
\label{lem:Phimonotone}
For every $k \in \mathbb{N}$, the map $\bphi^{(k)}(\bz)$ is strongly monotone over $\cX^{(k)} \times \cD^{(k)}$ with constant $\eta_{\phi} := \min\{p, d\}$ and Lipschitz over $\cX^{(k)} \times \cD^{(k)}$ with coefficient $L_\phi$ given by:
\begin{align}
L_\phi := \sqrt{(L + p + M_g + \xi_\lambda L_g)^2 + (M_g + d)^2}
\end{align}
where $L_g := \sqrt{\sum_{m = 1}^{M_I} L_{g_m}^2}$ and $\xi_\lambda := \sup_k \max_{\blambda \in \cD^{(k)}} \|\blambda\|_2$. 
\end{lemma}

The Lemma is a slight modification of~\cite[Lemma 3.4]{Koshal11}; the proof follows  steps that are similar to~\cite{Koshal11}. 

The proofs of Theorem \ref{theorem.inexact} and Theorem~\ref{cor:asym} are provided next. The proofs  follow steps that are similar to the ones outlined in~\cite{opfPursuit}; a summary of the steps as well as   modifications relative to~\cite{opfPursuit} are provided  for completeness.

\begin{proof}[Proof of Theorem \ref{theorem.inexact}]  Start from the the following equation:
\begin{align}
& \|\bz^{(k)} - \bz^{(*,k-1)} \|_2 = \|\proj_{\cX^{(k)} \times \cD^{(k)}} \{\bz^{(k-1)} - \hat{\bphi}^{(k)}(\bz^{(k-1)})\} \nonumber \\
& \hspace{5.0cm} - \bz^{(*,k-1)}\|_2 \, .\label{eq:proof_thm5_01}
\end{align}
Noticing that $\bz^{(*,k-1)}$ satisfies a fixed-point equation, leveraging the non-expansiveness property of the projection operator,   the following  inequality can be considered: 
\begin{align}
& \hspace{-.3cm} \|\bz^{(k)} - \bz^{(*,k-1)} \|_2 \leq \| \bz^{(k-1)}  - \alpha \hat{\bphi}^{(k-1)}(\bz^{(k-1)}) \nonumber \\
& \hspace{2.3cm} - \bz^{(*,k-1)} +  \alpha \bphi^{(k-1)}(\bz^{(*,k-1)})  \|_2 \, . \label{eq:proof_thm5_1}
\end{align}
Adding and subtracting $\bphi^{(k-1)}(\bz^{(k-1)})$ on the right-hand-side of~\eqref{eq:proof_thm5_1}, and using the triangle inequality, it follows that~\eqref{eq:proof_thm5_1} can be further bounded as:
\begin{align}
& \|\bz^{(k)} - \bz^{(*,k-1)} \|_2 \leq  \alpha \|\bphi^{(k-1)}(\bz^{(k-1)}) - \hat{\bphi}^{(k-1)}(\bz^{(k-1)})\|_2 \nonumber \\
& + \|\bz^{(k-1)}  - \alpha \bPhi^{(k-1)}(\bz^{(k-1)}) - \bz^{(*,k-1)} +  \alpha \bPhi^{(k)}(\bz^{(*,k-1)})  \|_2 \, . \label{eq:proof_thm5_2}
\end{align}
Following~\cite{opfPursuit}, using the results of Lemma~\ref{lem:Phimonotone}, the second term on the right-hand-side of~\eqref{eq:proof_thm5_2} can be bounded with the term $c(\alpha) \|\bz^{(k-1)} - \bz^{(*,k-1)}\|_2$; therefore,
\begin{align}
& \|\bz^{(k)} - \bz^{(*,k-1)} \|_2 \leq  \alpha \|\bphi^{(k-1)}(\bz^{(k-1)}) - \hat{\bphi}^{(k-1)}(\bz^{(k-1)})\|_2 \nonumber \\
& \hspace{1.5cm} + c(\alpha) \|\bz^{(k-1)} - \bz^{(*,k-1)}\|_2 . \label{eq:proof_thm5_3}
\end{align}
Consider now bounding $\|\bz^{(k)} - \bz^{(*,k)}\|_2$ as follows:
\begin{align}
\|\bz^{(k)} - \bz^{(*,k)}\|_2 & = \|\bz^{(k)} - \bz^{(*,k-1)} + \bz^{(*,k-1)} - \bz^{(*,k)} \|_2 \nonumber \\
& \hspace{-2.2cm} \leq \|\bz^{(*,k-1)} - \bz^{(*,k)} \|_2  + \|\bz^{(k)} - \bz^{(*,k-1)}\|_2  \label{eq:proof_thm5_4} \\
& \hspace{-2.2cm}  \leq \bar{\sigma}^{(k)}  + \alpha \|\bphi^{(k-1)}(\bz^{(k-1)}) - \hat{\bphi}^{(k-1)}(\bz^{(k-1)})\|_2  \nonumber \\
& \hspace{-2.0cm} + c(\alpha) \|\bz^{(k-1)} - \bz^{(*,k-1)}\|_2 . \label{eq:proof_thm5_5} 
\end{align}
By recursively applying~\eqref{eq:proof_thm5_5}, the result of Theorem \ref{theorem.inexact} follows.

\end{proof}

\begin{proof}[Proof of Theorem \ref{cor:asym}]  
  To show~\eqref{eqn:asym_bound_apr}, utilize the results of Lemma~\ref{lem:errorPhi} to bound $\|\bphi^{(k)}(\bz^{(k-1)}) - \hat{\bphi}^{(k)}(\bz^{(k-1)})\|_2$ and leverage the definition of $\bar{\sigma}$. Then,~\eqref{eq:proof_thm5_5} can be bounded as:
\begin{align}
& \|\bz^{(k)} - \bz^{(*,k)}\|_2 \nonumber \\
& \hspace{.2cm} \leq c(\alpha) \|\bz^{(k-1)} - \bz^{(*,k-1)}\|_2 + \bar{\sigma} + \alpha \sqrt{e_p^2 + e_d^2} \label{eq:proof_thm5_6}
\end{align}
since $c(\alpha) < 1$,~\eqref{eq:proof_thm5_6} represents a contraction. The result~\eqref{eqn:asym_bound_apr} can then be obtained  via the
geometric series sum formula. 
\end{proof}

\bibliographystyle{IEEEtran}
\bibliography{biblio.bib,biblio1.bib}

\end{document}